\pdfoutput=1
\documentclass[11pt]{amsart}
\usepackage{geometry}                
\usepackage[parfill]{parskip}   
\usepackage{graphicx,color,amsthm}
\usepackage{amssymb}
\usepackage{epstopdf}
\usepackage{pdfsync}
\usepackage{caption}
\usepackage{subcaption}
\usepackage{mathtools}
\usepackage{fullpage}
\theoremstyle{plain} 
\newtheorem{thm}{Theorem}

\newtheorem{cor}[thm]{Corollary} 
\newtheorem{lem}[thm]{Lemma} 
\newtheorem{prop}[thm]{Proposition}

\theoremstyle{definition} 
\newtheorem{defn}{Definition}
\newtheorem{ex}{Example}

\title{Singular branched covers of four-manifolds}
\author{Patricia Cahn and Alexandra Kjuchukova}
\thanks{This work was partially supported by the Simons Foundation/SFARI (Grant Number 523862, P. Cahn)}
\begin{document}

\begin{abstract} 
Consider a dihedral cover $f: Y\to X$ with  $X$ and $Y$ four-manifolds and $f$ branched along an  oriented surface  embedded in $X$ with isolated cone singularities. We prove that only a slice knot can arise as the unique singularity on an irregular dihedral cover $f: Y\to S^4$ if $Y$ is homotopy equivalent to $\mathbb{CP}^2$ and construct an explicit infinite family of such covers with $Y$ diffeomorphic to $\mathbb{CP}^2$.  An obstruction to a knot being homotopically ribbon arises in this setting, and we describe a class of potential counter-examples to the Slice-Ribbon Conjecture. 

Our tools include lifting a trisection of a singularly embedded surface in a four-manifold $X$ to obtain a trisection of the corresponding irregular dihedral branched cover of $X$, when such a cover exists. We also develop a combinatorial procedure to compute, using a formula by the second author, the contribution to the signature of the covering manifold which results from the presence of a singularity on the branching set. \end{abstract}
\maketitle
MSC classes: 57M27 (Primary), 57M25 and 57M12 (Secondary).

\section{Introduction}

Given a four-manifold $X$, what four-manifolds can be realized as branched covers of $X$? We approach this question by relating invariants of the covering manifold to invariants of the branching set. Our focus is on irregular dihedral covers. Following the set-up of~\cite{kjuchukova2016classification}, we consider branching sets which are closed oriented singularly embedded surfaces; the singularities considered are cones on  knots. 

Our first theorem classifies irregular dihedral covers $f: \mathbb{CP}^2\to S^4$. This result provides a roadmap to search for counter-examples to the Slice-Ribbon Conjecture, using a knot invariant, $\Xi_p$, defined in~\cite{kjuchukova2016classification}  for any knot which arises as a singularity on a $p$-fold irregular dihedral cover between any two four-manifolds.  If a knot $\alpha$ arises as the unique singularity on a $p$-fold irregular dihedral cover $f: Y\to S^4$ with $Y$ a manifold, we say $\alpha$ is {\it $p$-admissible}. We prove that a 3-admissible knot $\alpha$ with $|\Xi_3(\alpha)|\neq 1$ can not be homotopically ribbon. Therefore, evaluating $\Xi_3(\alpha)$ for admissible slice singularities $\alpha$ could lead to finding a non-ribbon slice knot. On the other hand, if it turns out that $|\Xi_3(\alpha)|= 1$ for all admissible slice singularities $\alpha$, then $\Xi_3$ provides a potentially new sliceness obstruction. We also derive a more general homotopy ribbon obstruction using $\Xi_p$.

The homotopy ribbon obstruction arising from the signature defect can be extended to a larger class of knots, call them {\it rationally $p$-admissible}. These are knots satisfying the first and third criterion for admissibility outlined in Section~\ref{criteria} but whose $p$-fold dihedral covers are rational homology spheres, rather than necessarily three-spheres. A rationally $p$-admissible knot $\alpha$ would thus arise as a singularity on a dihedral cover $Y\to X$ with $X$ a manifold and $Y$ a rational Poincar\'e Duality space with a singular point $\mathfrak{z}$. (The link of the singularity $\mathfrak{z}\in Y$  is the dihedral cover of $\alpha$.) Ribbon obstructions for rationally $p$-admissible knots will be the subject of future work.

In Theorem~\ref{procedurethm} we give a combinatorial procedure for calculating $\Xi_p(\alpha)$ from a diagram of $\alpha$, using the formula provided in~\cite{kjuchukova2016classification} as well as the algorithm developed in~\cite{cahnkjuchukova2016linking}. In addition to its purely knot-theoretic interest, this procedure for evaluating $\Xi_p$ allows us to compute signatures of dihedral covers of four-manifolds with singular branching sets.

$D_p$ denotes the dihedral group of order $2p$. In this paper, $p$ is odd.

\begin{defn} 
\label{dih}
Let $X$ be a manifold and $B\subset X$ a codimension-two submanifold with the property that a surjection $\phi: \pi_1(X-B, x_0) \to D_p$ exists. Denote by $\mathring{Y}$ the covering space of $X-B$ corresponding to the subgroup $\phi^{-1}(\mathbb{Z}/2\mathbb{Z})\subset \pi_1(X-B, x_0)$. The completion of $\mathring{Y}$  to a branched cover $f: Y\to X$ is called the {\it irregular dihedral} $p$-fold cover of $X$ branched along $B$.
\end{defn}

Let us comment briefly on why we choose to study these covers. First, they are the subject of Hilden~\cite{hilden1974every} and Montesinos's~\cite{montesinos1974representation} strikingly general result that every closed oriented three-manifold is a { three-fold} cover of $S^3$ branched along {a knot}. How well does this result generalize to the next dimension? Secondly, irregular dihedral covers of four-manifolds provide a rich source of examples due to the fact that they allow manifolds to cover manifolds even when the branching sets are singularly embedded surfaces; Section~\ref{admiss} offers a discussion of the singularities that can arise on these embeddings. Thirdly, methods for analyzing irregular dihedral covering maps between manifolds were developed by the second author, who gave a formula for the signature of the covering manifold in terms of data about the base and the singularities on the branching set~\cite{kjuchukova2016classification}. The tools developed therein allow us to characterize irregular dihedral covers $f: \mathbb{CP}^2\to S^4$ and also to construct an infinite family of examples of such covers. Our proof that the manifolds constructed are homeomorphic to $\mathbb{CP}^2$ relies on determining their intersection forms; we verify this conclusion independently by obtaining trisections of the branched covers constructed. 

The singularities we consider on branching sets are defined next. Denote by $\alpha\subset S^3$ a non-trivial smooth knot.

\begin{defn} 
\label{sing}
Let $X$ be a topological four-manifold, $B\subset X$ be a properly embedded surface, and $z\in B$ a point on the interior of $B$. Assume there exist a small open disk $D_z$ about $z$ in $X$ such that there is a homeomorphism of pairs $(D_z - z,B - z) \cong (S^3 \times (0,1),\alpha\times(0,1))$. We say that the embedding of $B$ in $X$ has a {\it singularity of type $\alpha$} at $z$.
\end{defn}

We are then interested in irregular dihedral covers $f: Y\to X$, where $X$ and $Y$ are closed oriented four-manifolds, and such that the branching set of $f$ is a closed oriented connected surface, smoothly embedded in $X$ except for finitely many singularities in the above sense. Given such a map $f$ and a locally flat point $b$ on the branching set of $f$, we require that $f$ admit a parametrization as a smooth branched cover over a small neighborhood  $D_b$ of $b$ in $X$. We further assume that $f^{-1}(\partial D_z)$ is connected and that the restriction and $f_|: {f^{-1}(D_z)}\to D_z$ is the cone on the map $f_|: {f^{-1}(\partial D_z)}\to \partial D_z$. We call such a map $f$ a {\it singular dihedral branched cover}. 

Given a singular dihedral branched cover $f: Y\to X$ between four-manifolds, the second author gave a formula for the signature of $Y$ in terms of data about $X$, the branching set and the singularities of the embedding~\cite{kjuchukova2016classification}. The contribution to this signature resulting from the presence of each singularity $\alpha$ is the quantity $\Xi_p(\alpha)$ introduced previously. We call the integer $\Xi_p(\alpha)$ the {\it signature defect} associated to the knot type $\alpha$. It is defined whenever the knot $\alpha$ arises as a singularity on a dihedral cover between four-manifolds (see Section~\ref{admiss}). By Proposition~2.7 of~\cite{kjuchukova2016classification}, if an admissible singularity $\alpha$ admits only one equivalence class of surjective homomorphisms to the dihedral group $D_p$, $\Xi_p(\alpha)$ is an invariant of $\alpha$. A straight-forward generalization shows that, if $\alpha$ admits multiple such surjections, $\Xi_p(\alpha)$ is an invariant of $\alpha$, together with a choice of coloring. A combinatorial procedure for calculating $\Xi_p(\alpha)$ from a colored diagram of $\alpha$ is outlined in the next section and illustrated on two examples in Section~\ref{examples}.

\begin{thm} 
\label{manyCP2s}
Let $f: Y \to S^4$ a $p$-fold singular dihedral branched cover in the above sense, with $Y$ an oriented manifold homotopy equivalent to $\mathbb{C}P^2$ or $\overline{\mathbb{C}P}^2$. Denote the number of singular points by $m$ and the genus of the branching set by $g$. Then $p=3$ and $g=\frac{m-1}{2}$. Moreover, for $Y$ diffeomorphic to $\mathbb{C}P^2$ or $\overline{\mathbb{C}P}^2$, an infinite family of such covers with $g=0$ and $m=1$ exists, branched along two-spheres embedded in $S^4$ with pairwise distinct two-bridge singularities. 
\end{thm}

 \begin{cor}  Let $f: Y \to S^4$ be a $p$-fold singular dihedral branched cover with $Y$ an oriented manifold homotopy equivalent to $\mathbb{C}P^2$ or $\overline{\mathbb{C}P}^2$. If the branching set of $f$ has only one singularity $\alpha$, then $\alpha$ is a slice knot. 
 	\end{cor}

The above theorem can be regarded as a classification of singular dihedral covers $f: \mathbb{CP}^{2} \to S^4$ in terms of their degree, branching set and number of singularities. Since sliceness is a necessary condition for a knot to occur as the only singularity on such a cover, it is natural to ask: which slice knots arise in this context?

\begin{defn} 
\label{hom-rib}
Let $\alpha\subset S^3$ be a slice knot and  $D\subset B^4$ a slice disk for $\alpha$. If the map $\iota_\ast: \pi_1(S^3-\alpha, x_0)\to \pi_1(B^4-D, x_0)$ induced by inclusion is surjective, we say that $D$ is a {\it homotopically ribbon} disk. A knot which admits such a disk is called a {\it homotopically ribbon} knot.
\end{defn}

We use dihedral covers and the signature defect to derive an obstruction to a knot being homotopically ribbon. 

\begin{thm} 
\label{homoCP2}
Let $Y$ be a closed oriented connected topological four-manifold and $f: Y \to S^4$ a $3$-fold irregular dihedral branched cover with branching set a two-sphere $B$ embedded in $S^4$ with one singularity of type $\alpha$. Assume further that $\alpha$ is a smoothly homotopically ribbon knot and that $B$ is the boundary union of a smooth homotopically ribbon disk $D$ for $\alpha$ and the cone on $\alpha$. Then $Y$ is a smooth manifold homeomorphic to $\mathbb{CP}^{2}$. If  $\alpha$ is topologically homotopically ribbon and $D$  locally flat, then $Y$ has the homotopy type of $\mathbb{CP}^{2}$. 
\end{thm}

In contrast with Theorem~\ref{manyCP2s}, where we use a particular family of singularities, the above theorem does not determine the diffeomorphism classes of the smooth manifolds constructed.  Note also that one could potentially obtain a fake (non-smooth) $\mathbb{CP}^2$ as a branched cover of the four-sphere using a singularity which is topologically but not smoothly slice. 

\begin{thm} 
\label{ribbon-defect}
Let $\alpha$ be a homotopically ribbon knot. If $\Xi_p(\alpha)$ is defined, it satisfies the equation $|\Xi_p(\alpha)|
\leq (p-1)/2$.  In particular, given $f: Y\to X$ a 3-fold singular irregular dihedral cover whose branching set $B$ has a single singularity $\alpha$, $|\Xi_3(\alpha)|=1$, and
\begin{equation}
\label{Xi=1}
\sigma(Y)=3\sigma(X) -\frac{p-1}{4} e(B) \pm{1}, 
\end{equation}
where $e(B)$ denotes the normal Euler number of the embedding of $B$ in $X$. 
\end{thm}

A quick remark on the sign of the last term in the above formula. It is evident that a knot arises as a singularity on a dihedral cover if and only if its mirror image does. Moreover, it follows directly from Equation~\ref{eqXi} that taking the mirror of $\alpha$ reverses the sign of the signature defect $\Xi_p(\alpha)$. With these considerations in mind, we occasionally use $\alpha$ to denote both mirror images of a knot. This convention leaves the sign of the defect term in the above formula ambiguous. 

\begin{cor} 
\label{twobridge}
Let $\alpha$ be a $p$-colorable two-bridge slice knot. Then $\Xi_p(\alpha)$ is defined, and $|\Xi_p(\alpha)|\leq (p-1)/2$.  In particular, if $p=3$, then $|\Xi_3(\alpha)|=1$, and in the notation of Theorem ~\ref{ribbon-defect}, Equation~\ref{Xi=1} holds. 
\end{cor}

We summarize the knot-theoretic questions motivated by the above results. First,  for $\alpha$ a slice knot such that $\Xi_3(\alpha)$ is defined, does the equality $|\Xi_3(\alpha)|=1$ always hold?  More generally, if $\alpha$ is slice and $p$-admissible, does the inequality $|\Xi_p(\alpha)|\leq (p-1)/2$ always hold? If the answer is no, the Slice-Ribbon Conjecture is false. If the answer is yes,  $\Xi_p(\alpha)$ provides a sliceness obstruction.  In the latter case, we ask further: for $\alpha$ and $\beta$ concordant knots with $\Xi_p(\alpha)$ and $\Xi_p(\beta)$ defined, does the equality $|\Xi_p(\alpha)| = |\Xi_p(\beta)|$ hold?

Evaluating the invariant $\Xi_p$ is therefore of interest both for computing signatures of singular branched covers and for its applications to knot concordance. Theorem~\ref{procedurethm}, stated in the next section, outlines a combinatorial procedure for computing this signature defect.

\section{Admissible singularities and the signature of a branched cover}
\label{admiss}

Let $f: Y\to X$ be a $p$-fold singular dihedral branched cover, with $X$ and $Y$ closed oriented four-manifolds. Denote by $B$ the (oriented) branching set of $f$  and by $\alpha$ a knot that arises as a singularity type on the embedding of $B$ in $X$.  As noted earlier, the signature defect  $\Xi_p(\alpha)$ is defined in this context. 
If, in addition, $X$ is the four-sphere, $\alpha$ is what we called a {$p$-admissible} knot.  Understanding admissible singularities is a necessary step for classifying dihedral covers between four-manifolds and computing their signatures, as well as for using the obstruction to being homotopically ribbon given in Corollary~\ref{ribbon-defect}. In this section, we give a necessary and sufficient condition for $p$-admissibility. This condition consists of three criteria; the first two are purely local and need to be satisfied by any singularity $\alpha$ on a $p$-fold dihedral cover between four-manifolds. The third criterion stems from the the additional assumption that the base be $S^4$, and analogous criteria can be defined for other manifolds. We conclude the section by describing a combinatorial procedure for computing $\Xi_p(\alpha)$ from a knot diagram.

\subsection{Three criteria for $p$-admissibility}\label{criteria}
Assume the notation of Definition~\ref{sing}. By our definition of a singular dihedral cover $f$, we have $f^{-1}(\partial D_z)$ connected. This gives the first criterion for $p$-admissibility of a knot $\alpha$: the sphere $S^3$ must  admit a $p$-fold irregular dihedral cover branched along $\alpha$. If $p$ is square-free, it suffices to check that $p$ divides the determinant of $\alpha$. For a general $p$, the existence of such a cover is equivalent to saying that the group $\pi_1(S^3-\alpha, x_0)$ surjects to the dihedral group $D_p$.  Fox's $p$-colorability is a combinatorial approach to detect the existence of such a surjection. In particular, two homomorphisms $\pi_1(S^3-\alpha, x_0)\to D_p$, or two Fox colorings, are called {\it equivalent} if they are related by an automorphism of $D_p$. Equivalent colorings induce homeomorphic dihedral covers. The existence of a dihedral branched cover of $\alpha$ can also be stated using the following notion.
 
\begin{defn}
\label{chark}
Let $\alpha\subset S^3$ be a knot and $V$ a Seifert surface for $\alpha$ with Seifert form $L$. Let $\beta\subset V^\circ$ be an embedded curve which represents a primitive class in $H_1(V; \mathbb{Z})$. If $L(\beta, \omega) + L(\omega, \beta) \equiv 0\mod p$ for all curves $\omega$ representing non-zero classes in  $H_1(V; \mathbb{Z})$, we say that $\beta$ is a {\it mod~p characteristic knot} for $\alpha$.
\end{defn}

The existence of  a $p$-fold irregular dihedral cover of $S^3$ branched along $\alpha$ is equivalent the existence of a {\it mod~p characteristic knot} for $\alpha$~\cite{CS1984linking}. Also see Section~\ref{dih-con}. Characteristic knots play a key role in computing the contribution $\Xi_p$ to the signature of a branched cover arising from the presence of a singularity. 

We have seen that a knot which arises as a singularity on a $p$-fold dihedral cover between four-manifolds must itself admit a $p$-fold irregular dihedral cover. The second criterion such a knot must satisfy has to do with the homeomorphism type of this cover. Given a $f$ as above and a singularity $z$ of type  $\alpha$ on the branching set of $f$, denote by $M$ the irregular dihedral $p$-fold cover of  $\alpha$ determined by $f$. As before, $D_z$ denotes a neighborhood of $z$ in $X$. By definition of a singular dihedral cover, $f^{-1}(D_z)\subset Y$ is the cone on $M$. Since $Y$ is a manifold, $M$ must be the three-sphere. 

It is a classical result that all two-bridge knots have this property any (odd) $p$. The proof of this fact amounts to computing the Euler characteristic of the lift to $M$ of a bridge sphere for $\alpha$. When $\alpha$ is two-bridge, this lift is seen to give a genus-0 Heegaard splitting of $M$. Thus, $p$-colorable two-bridge knots are for us a rich source of examples. Infinite families of three-bridge knots whose {three-fold} dihedral covers are $S^3$ have been identified -- see, for example~\cite{greene2011slice}. Determining the homeomorphism type of a dihedral cover of a knot, for example using~\cite{fox1970metacyclic}, becomes increasingly complicated as the bridge index of the knot and the degree of the cover grow. We did not find in the literature any general method for passing between a dihedral branched cover representation of a closed oriented three-manifold and a Heegaard diagram, so we devised a procedure to do this by hand -- see Section~\ref{sing-trisections}. Our immediate purpose was used to identify the families of dihedral covers constructed in the proof of Theorem~\ref{manyCP2s} via trisection diagrams; however, the same procedure can be applied to search for admissible singularities which are not two-bridge.

The third criterion for $p$-admissibility is not purely local but, rather, may depend on the base of the branched cover. When the base is $S^4$, this criterion captures the fact that  the $p$-fold irregular dihedral cover of $\alpha$ bounds a dihedral cover of the four-ball,   branched along properly embedded surface with boundary $\alpha$. 
Observe that this condition is satisfied by every $p$-colorable knot $\alpha$ if we allow singularly embedded surfaces in the four-ball with boundary $\alpha$, since the dihedral cover of a knot always extends over the cone on the knot.  Thus, we only consider locally flatly embedded surfaces in the four-ball, which correspond to covers of $S^4$ which have one singularity (each) on the branching sets. 
  
Let us return for a moment to the case of singular dihedral branched covers of an arbitrary four-manifold $X$. Denote by $\hat{X}$ the complement in $X$ of a neighborhood of the singularity. This last criterion can be cast in terms of the existence of a surface $F$ embedded in $\hat{X}$ so that the surjective homomorphism $\phi: \pi_1(S^3 - \alpha)\to D_p$ extends to a homomorphism $\phi: \pi_1(\hat{X} - F)\to D_p$.   If $\alpha$ is a slice knot, by Lemma~3.3 of~\cite{kjuchukova2016classification}, $F$ can always be chosen to be a slice disk contained in a four-ball properly embedded in $\hat{X}$. For a non-slice knot, the existence of a surface $F$ that admits such an extension may conceivably depend on the ambient manifold $X$. 

In the case where X is $S^4$, the second author and Kent Orr have found an obstruction to the existence of a surface $F$ as above and showed that the obstruction is sharp~\cite{kjorr2017admissible}. This allows for a complete classification of admissible two-bridge singularities over $S^4$ as the base and gives infinite families of $p$-admissible non-slice knots for all $p$.  In the current paper we use slice knots for our examples.

\subsection{The signature defect arising from a singularity}
In this section we give combinatorial procedure  for computing $\Xi_p(\alpha)$. This relies on the formula given in  Proposition~2.7 of~\cite{kjuchukova2016classification}, which we now recall. Let $\alpha$ be a $p$-admissible knot, $\beta$ the  characteristic knot corresponding to the relevant surjection  $\phi: \pi_1(S^3 - \alpha)\to D_p$, and $V$ the Seifert surface for $\alpha$ containing $\beta$ in its interior (see Definition~\ref{chark}). Furthermore, let $L_V$ denote the symmetrized Seifert form of $V$, $\zeta$ a primitive $p$-th root of unity, and $\sigma_{\zeta^i}$ the Tristram-Levine  $\sigma_{\zeta^i}$ signature. Then, 
\begin{equation}\label{eqXi}
\Xi_p(\alpha) = \frac{p^2-1}{6p}L_V(\beta, \beta)  + \sum_{i=1}^{p-1} \sigma_{\zeta^i}(\beta) + \sigma(W(\alpha, \beta))
\end{equation}

Here, $\sigma(W(\alpha, \beta))$ the denotes the signature of a four-manifold $W(\alpha, \beta)$ constructed by Cappell and Shaneson in~\cite{CS1984linking} and discussed in more detail in Section~\ref{octopus}. Remark that the first two terms in the above expression for $\Xi_p(\alpha)$ present no computational difficulty, while the calculation of $\sigma(W(\alpha, \beta))$ gets rather technical. Thus, we focus our attention on this term but postpone the definition of $W(\alpha, \beta)$. For the moment,  it suffices to know that  $\sigma(W(\alpha, \beta))$ can be computed in terms of linking numbers of certain curves  in the $p$-fold dihedral cover of $\alpha$. These curves are lifts to the dihedral cover of a basis $\mathcal{B}_V:=\{\beta_r,\beta_l, \omega_1, ..., \omega_{2g-2}\}$  for $H_1(V-\beta; \mathbb{Z})$, where $\beta_r$ and $\beta_l$ are push-offs of $\beta$, and the $\omega_i$ are curves in the interior or $V-\beta$. The relevant linking numbers are computed using the algorithm given in~\cite{cahnkjuchukova2016linking}. 

The purpose of the current section is to condense all this information in a labeled knot diagram of $\alpha$, $\beta$ and the $\omega_j$, so that the signature defect can be computed algorithmically. The resulting algorithm is the content of Theorem~\ref{procedurethm}.

The labeled link diagram we use is as follows.  One component of the link diagram is the knot $\alpha$. In order to simplify the combinatorics, we only include two of $\{\beta, \omega_1, ... \omega_{2g-2}\}$, or one of these curves together with its push-off in $V$, in our diagram at any given time.  Call these two curves $g$ and $h$.   Because $\beta$ is a mod 3 characteristic knot, any curve in $V-\beta$ lifts to three closed loops~\cite{CS1984linking}.  Thus for each pair of curves in $\mathcal{B}_V$, we compute nine linking numbers of their lifts, organized in a symmetric $3\times 3$ matrix.  The following set-up allows us to compute the intersection number of any lift of $h$ with a 2-chain whose boundary is any given lift of $g$.  For the details on how this 2-chain is constructed see \cite{cahnkjuchukova2016linking}.

{\it(1)} The arcs of $\alpha$ in the diagram $\alpha\cup g$ are labeled $0,1, \dots, m-1$, where $m$ is the number of self-crossings of $\alpha$ plus the number of crossings of $\alpha$ under $g$.  Each arc of $\alpha$ is colored 1,2 or 3, according to the given homomorphism $\rho:\pi_1(S^3-\alpha)\twoheadrightarrow D_3$.

{\it(2)} The arcs of $g$ in the diagram $\alpha\cup g $ are labeled $0,1,\dots, n-1$, where $n$ is the number of self-crossings of $g$ plus the number of crossings of $g$ under $\alpha$.

{\it(3)} Now we add $h$ to the above numbered diagram $\alpha\cup g$ without changing the numbering of any existing arcs.   The arcs of $h$ are labelled $0,1,\dots, o-1$, where $o$ is the number of crossings of $h$ under $\alpha$ plus the number of crossings of $h$ under $g$.  In this article, $h$ never has self-crossings. 

In addition, we need the following combinatorial information about this diagram:
	
The irregular dihedral cover $M_\alpha$ corresponding to $\rho$ is equipped with a cell structure coming from the cone on $\alpha$.  See~\cite{cahnkjuchukova2016linking} for details.  Let $\omega_i^j$ denote the lift of $\omega_i$ such that the lift of its zeroth arc lies in the $j^{th}$ 3-cell, for $j=1,2,3$.  The lifts $\beta_r^j$ and $\beta_l^j$ are defined analogously. 

An {\it anchor path} for a curve $\omega\subset V-\beta$ is a properly embedded path $\delta$ in $V-\beta$ from a point $q$ on the zeroth arc of $\alpha$ to a point $r$ on the zeroth arc of $\omega$.  Suppose $\delta$ crosses under a set $a_1,\dots a_k$ of arcs of $\alpha$ in that order, as one traverses $\delta$ from $q$ to $r$.  The {\it monodromy} of the anchor path $\delta$ is the product of the permutations $\sigma_k\dots\sigma_2\sigma_1$,  where $\rho(a_i)=\sigma_i$ is the permutation associated to the arc $a_i$ of $\alpha$.

\begin{thm}\label{procedurethm} 
 Let $\alpha$ be a knot and $\rho:\pi_1(S^3-\alpha) \twoheadrightarrow D_3$ a homomorphism to the symmetric group on three elements.  Let $\mathcal{B}_V=\{\omega_i\}_{i=1}^{2g-2}\cup \{\beta_r,\beta_l\}$ be any basis for $H_1(V-\beta;\mathbb{Z})$ consisting of embedded curves in a Seifert surface $V$ for $\alpha$, where $\beta$ is a mod 3 characteristic knot for $\alpha$.  Let $\delta_i$ be an anchor path for $\omega_i$, and let $\gamma_r$ and $\gamma_l$ be anchor paths for the right and left pushoffs of $\beta$ in $V$.   Let $\mu_{\delta_i}, \mu_{\gamma_r},$ and $\mu_{\gamma_l}\in D_3$ be their monodromies.  Let $c_0\in \{1,2,3\}$ be the color of the zeroth arc of $\alpha.$  Then the signature of the matrix of linking numbers of the following curves
	
	$$\omega_i^j-\omega_i^k, \text{ where } \{j,k\}=\{1,2,3\}-\{\mu_{\delta_i}(c_0)\}$$
	$$\beta^j-\beta^k, \text{ where } j=\mu_{\gamma_r}(c_0) \text{ and }\{k\}=\{1,2,3\}-\{\mu_{\gamma_r}(c_0),\mu_{\gamma_l}(c_0)\}$$
is independent of the choices of anchor paths $\delta_i$, $\gamma_r$, and $\gamma_l$, and is equal to $\sigma(W(\alpha,\beta))$. 
Setting this value equal to $\sigma (W(\alpha,\beta))$ in  Equation~\ref{eqXi} yields the value of $\Xi_3(\alpha).$
	\end{thm}

In Section~\ref{examples}, we illustrate how to apply this theorem to compute the signature defect associated to a singularity. The first knot we use as an example is $6_1$, one of the singularities in the family constructed in the proof of Theorem~\ref{manyCP2s}. Our second example is a 3-admissible knot whose Seifert surface has higher genus, in order that the additional curves  $\omega_i$ and their anchor paths come into play.  In addition to using the above procedure to evaluate the signature of the cover, in Section~\ref{trilift} we outline a method which allows us to identify the covering manifold using trisections. This involves adapting a result on trisecting smooth surfaces in four-manifolds~\cite{meierzupan} to include the case of  surfaces with isolated cone singularities, and then lifting a trisection of the base manifold to a trisection of its dihedral cover. This process is explained in the next section.

	\section{Colored Singular Bridge Trisections and their Coverings}\label{trilift}

In this section we give a method to identify, via trisections, the manifold obtained by taking a singular dihedral branched cover in terms of data about the base and branching set.
First, we explain how to modify tri-plane diagrams of smoothly embedded surfaces in $S^4$~\cite{meierzupan} so as to trisect singular knotted surfaces in $S^4$.  Then we show that given a representation $\pi_1(S^4-B)\twoheadrightarrow D_p$, the trisection of the singular knotted branching surface $B$ lifts to a trisection of the covering 4-manifold.  We use this setup to construct an explicit infinte family of covers of $S^4$, branched along singular two-spheres with distinct isolated cone singularities, and show each is diffeomorphic to $\mathbb{C}P^2$.  A slight modification of this construction yields an infinite family of covers $S^4\rightarrow S^4$ whose singularities are the cone on a link rather than a knot.

\subsection{Bridge trisections of singular surfaces in $S^4$}  
\label{sing-trisections}

First we recall the definition of a trisection of a closed, connected, oriented four-manifold $Y$ due to Gay and Kirby~\cite{gaykirby2016trisections}.  Given $0\leq k_1,k_2,k_3\leq g$, a $(g;k_1,k_2,k_3)$-trisection of $Y$ is a decomposition of $Y$ into three submanifolds $Y_i$, $Y=Y_1\cup Y_2\cup Y_3$, such that

\begin{enumerate}
	\item $Y_i$ is diffeomorphic to $\natural^{k_i}(S^1\times B^3)$, and in particular is diffeomorphic to $B^4$ if $k_i=0$
	\item $Y_1\cap Y_2 \cap Y_3$ is a closed, orientable surface $F_g$ of genus $g$
	\item $Y_i\cap Y_j$ is a genus $g$ three-dimensional handlebody
	\end{enumerate}
	
The {\it spine} of a trisection is the union of the pairwise intersections $Y_i\cap Y_j$, and completely determines the trisection.  The spine is in turn determined by a Heegaard triple $(F_g,\mu_1,\mu_2,\mu_3)$ where the $\mu_i$ are $g$-tuples of curves on $F_g$.  Each pair of tuples $(F_g, \mu_i,\mu_{i+1})$ is a Heegaard diagram for the Heegaard splitting $\partial Y_i=(Y_i\cap Y_{i+1})^+\cup_{F_g}(Y_i\cap Y_{i-1})^-$ where indices are taken mod 3.
	
 Let $S^4=X_1\cup X_2 \cup X_3$ be the standard genus 0 trisection of $S^4$ given in~\cite{gaykirby2016trisections}.  This means that each $X_i$ is diffeomorphic to $B^4$, $X_i\cap X_j=\partial X_i \cap \partial X_j$ is homeomorphic to $B^3$, and $X_1\cap X_2\cap X_3 $ is homeomorphic to $S^2$.  

Now let $B$ be a non-singular knotted surface in $S^4$.  The notion of a trisection of $B$, introduced in~\cite{meierzupan}, can be viewed as a four-dimensional analogue of a bridge splitting of a knot in $S^3$.  A {\it trivial $c$-disk system} is a pair $(B^4,\mathcal{D})$ of $c$ properly embedded disks in $B^4$ which are simultaneously isotopic into the boundary of $B^4$.   A $(b;c_1,c_2,c_3)$-bridge trisection of $B$ is a decomposition of the pair $(S^4, B)=(X_1,\mathcal{D}_1)\cup (X_2,\mathcal{D}_2)\cup (X_3, \mathcal{D}_3)$ such that 

\begin{enumerate}
\item $S^4=X_1\cup X_2\cup X_3$ is the 0-trisection	of $S^4$
\item $(X_i,\mathcal{D}_i)$ is a trivial $c_i$-disk system
\item $(X_i,\mathcal{D}_i)\cap (X_j,\mathcal{D}_j)$ is a $b$-strand trivial tangle.
\end{enumerate}

\begin{defn}  An {\it $b$-strand trivial tangle} is a collection of arcs $\{t_1,\dots,t_b\}$, properly embedded in $B^3$, such that for each $i\in \{1,\dots,b\}$, there exists a path $d_i:[0,1]\rightarrow S^2=\partial B^3$ such that $d_i(0)=t_i(0)$, $d_i(1)=t_i(1)$, and their concatenation $d_i\cdot t_i$ is the boundary of a disk $\partial \Delta^2_i$, where the $\Delta^2_i$ are a collection of disjoint disks in $B^3$.  The disks $\Delta^2_i$ are called {\it bridge disks} and each $d_i$ is called the {\it shadow} of the arc $t_i$.
	\end{defn}

Given any bridge trisection of a non-singular knotted surface, the boundary of each disk system $\partial\mathcal{D}_i$ is a $c_i$-component unlink in $\partial X_i$.  On can represent a bridge trisection of $B$ combinatorially via a {\it tri-plane diagram} \cite{meierzupan}.  This is a triple of trivial tangle diagrams $A$, $B$, and $C$, such that $L_1=A\cup \bar{B}$, $L_2=B\cup \bar{C}$, and $L_3=C\cup \bar{A}$ are link diagrams for  $\partial \mathcal{D}_1$, $\partial\mathcal{D}_2$, and $\partial\mathcal{D}_3$, respectively, where the bar denotes the mirror image, and $L_i$ is a $c_i$-component unlink.

\subsection{Singular tri-plane diagrams and their colorings} A tri-plane diagram determines a bridge-trisected, knotted, non-singular embedded surface in $S^4$. Each unlink $L_i$ bounds a family of $c_i$ disks in $\partial X_i$.  Pushing these disks into the four-balls $X_i$ and taking their boundary union yields a $(b;c_1,c_2,c_3)$-trisected surface. We now extend the collection of tri-plane diagrams to allow for surfaces with singularities.  A {\it singular tri-plane diagram} is a triple of trivial tangle diagrams $(A,B,C)$, such that $L_1=A\cup \bar{B}$, $L_2= B\cup \bar{C}$, and $L_3=C\cup \bar{A}$ are $c_i$-component link diagrams, not necessarily of the unlink.

To construct a singular surface in $S^4$ from a singular tri-plane diagram, we again consider the links $L_1=A\cup \bar{B}$, $L_2=B\cup \bar{C}$, and $L_3= C\cup \bar{A}$.  If $L_i$ is an unlink, it bounds a collection of disks in the $\partial X_i$ as before.  Let $\mathcal{D}_i$ denote the push-off of these disks into $X_i$.  If $L_i$ is not an unlink, we take $\mathcal{D}_i$ to be the cone on $L_i$.  When the branching set has a single isolated singularity modelled on the cone on a knot $\alpha$, one of the $L_i$ will be isotopic to $\alpha$; the other two $L_i$ will be  $c_i$-component unlinks.  When each tangle has $b$ strands, call the corresponding trisection a $(b; 1; c_2, c_3)$ {\it singular bridge trisection}.  Alternatively, if one of the $L_i$ is a split link, one might want to define $\mathcal{D}_i$ to be the disjoint union of the cones on each of its components.  This would allow for multiple isolated singularities, each modelled on the cone on a knot, but we do not use that construction in this paper. 
 
\begin{figure}[h]\includegraphics[width=4in]{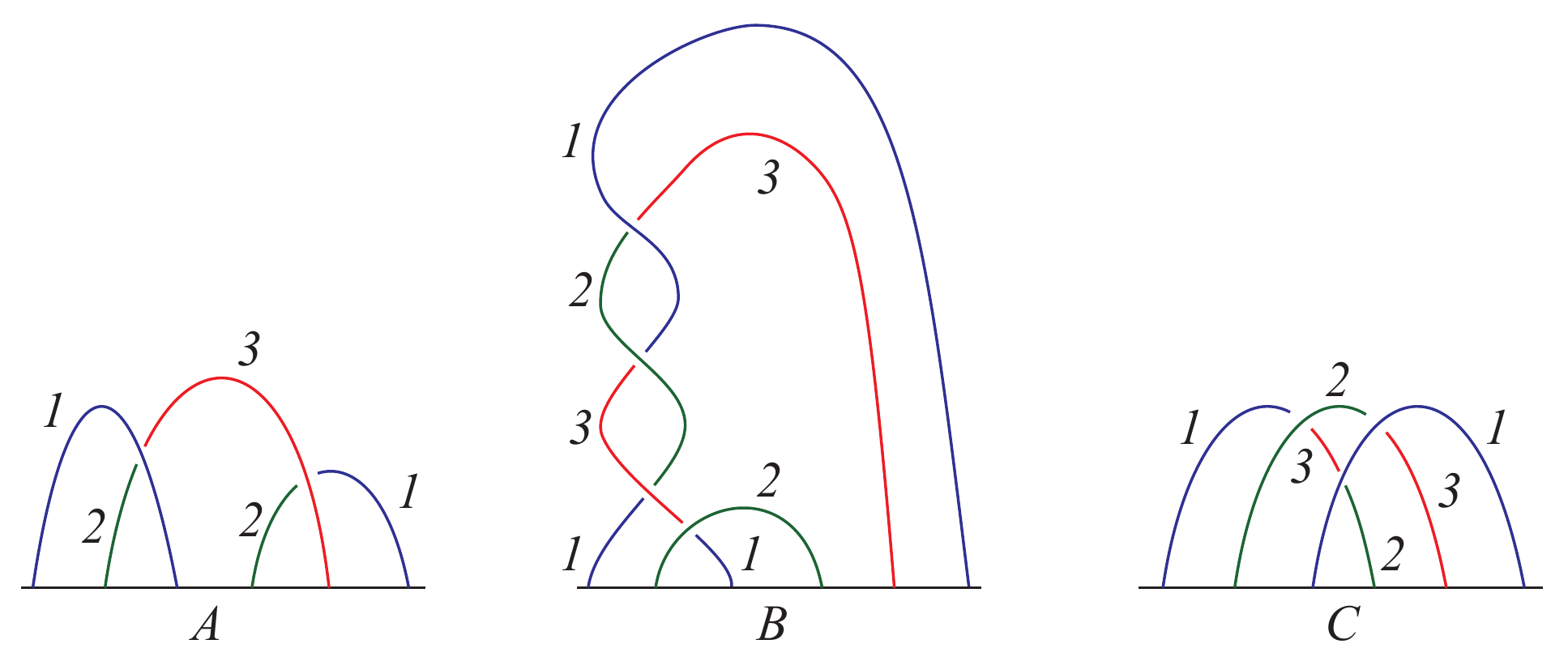}
	\caption{A tri-plane diagram for a $(3; 1, 2,2)$-bridge trisection of a two-sphere with one singularity of type $6_1$.}
	\label{triplanediagram61.fig}
	\end{figure}

A tri-plane diagram for a singular $(3;1,2,2)$-bridge trisection of a two-sphere in $S^4$ with a single cone singularity of type $6_1$ is pictured in Figure \ref{triplanediagram61.fig}.  Note that $A\cup \bar{B}$ is the knot $6_1$, while $B\cup \bar{C}$ and $C\cup \bar{A}$ are two-component unlinks.

A $p$-colored tri-plane diagram $(A,B,C)$ is a choice of $p$-colorings of the tangles $A$, $B$ and $C$ which induces a valid Fox $p$-coloring of  each of the link diagrams $A\cup \bar{B}$, $B\cup \bar{C}$, and $C\cup \bar{A}$. We require that least one of the induced Fox colorings of the link diagrams be non-trivial. For $p$ prime, it suffices to require that at least two colors are used. Note that a $p$-coloring on a tri-plane diagram determines a surjection to $D_p$ from the group of each of the knots among  $A\cup \bar{B}$, $B\cup \bar{C}$, and $C\cup \bar{A}$ for which the induced Fox coloring is non-trivial. 

Given a $p$-colorable knot $\alpha$, we define its {\it $p$-dihedral bridge number}, $b_p(\alpha)$, to be the minimum $b$ such that $\alpha$ arises as the boundary of the singular disk in a $p$-colored tri-plane diagram for a $(b;1,c_2,c_3)$-bridge trisection.

\begin{lem} Let  $(A_1, A_2, A_3)$  be a $p$-colored tri-plane diagram for a surface $F$ in $S^4$. This coloring determines a surjective homomorphism $\pi_1(S^4-F)\to D_p$ which extends the homomorphism $\pi_1(S^3-A_i\cup \bar{A_j})\to D_p$ given by the induced Fox coloring on the pair $A_i\cup \bar{A_j}$ for each $i\neq j$ in $\{1, 2, 3\}$.	\end{lem}
	\begin{proof} The intersection $\mathcal{D}_{ij}$ of $F$ with  each of the three four-balls $X_i$ that make up the trisection is a union of trivial two-disks and cones. Therefore, $\pi_1(S^3-A_i\cup \bar{A_j})$ surjects onto $\pi_1(B^4-\mathcal{D}_{ij})$.	The statement then follows from Van Kampen's Theorem, together with the definition of $p$-coloring.	\end{proof}

\subsection{Lifting a colored singular bridge trisection to a four-manifold dihedral cover} Given a homomorphism $\rho:\pi_1(S^4-B)\twoheadrightarrow D_{p}$, a bridge trisection of the branching set $B\subset S^4$ lifts to a trisection of the $p$-fold irregular dihedral cover $Y$ of $S^4$ branched along $B$.

\begin{figure}[htbp]\includegraphics[width=4in]{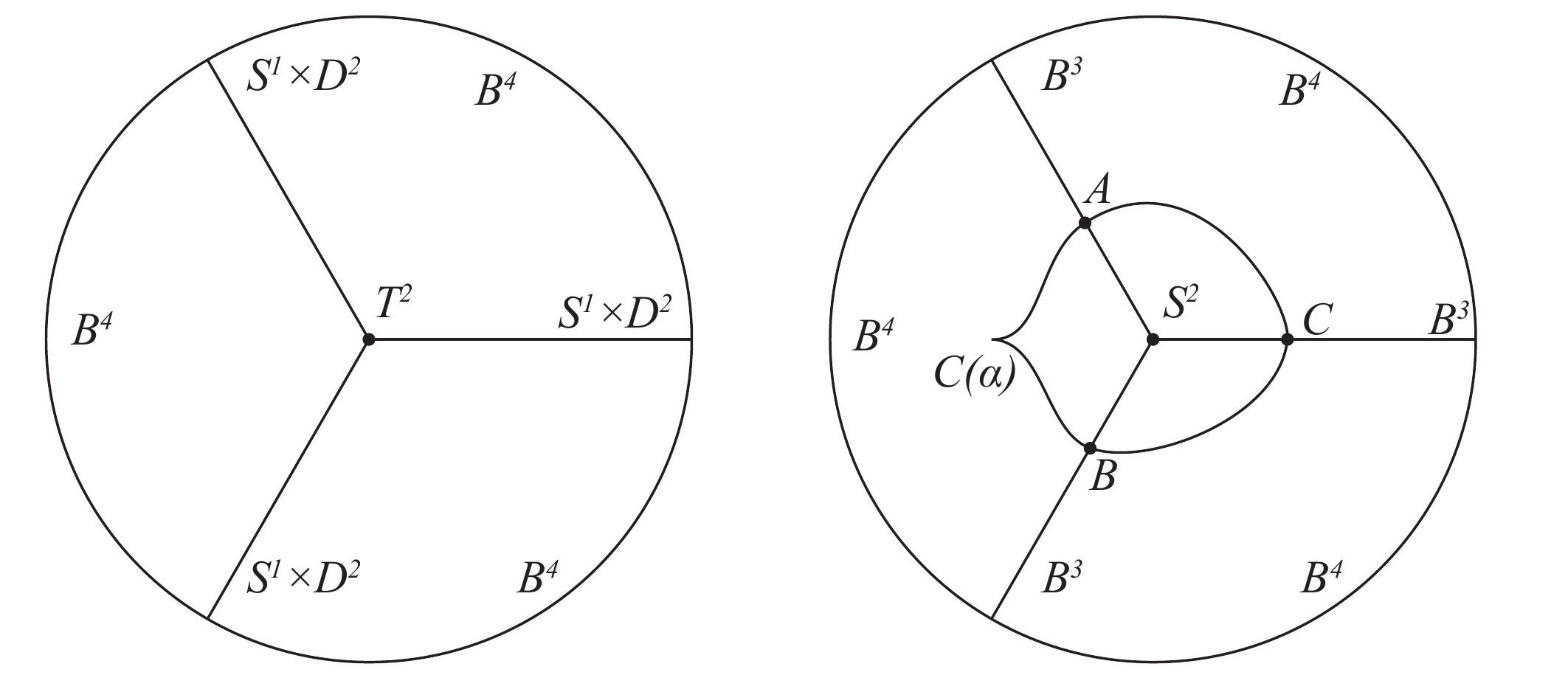}
	\caption{A trisection of the 4-manifold $Y$ (left) covering the bridge trisection of the singular knotted surface with one singularity of type $\alpha=6_1$ (right).}
	\label{trisections.fig}
	\end{figure}

\begin{thm} The $p$-fold irregular dihedral cover of $S^4$ branched along a surface $B
\subset S^4$ with one 
singularity, presented as a singular $(b; 1, c_2, c_3)$-bridge trisection, is naturally equipped with a $(g;k_1,k_2,k_3)$-trisection, where $g=1-p(1-b)-b(1+\frac{p-1}{2})$ , $k_1=0$, and for $i=2,3$, $k_i=1-p(1-c_i)-c_i-c_i(\frac{p-1}{2})$. \label{trisectionlift.thm}
\end{thm}

\begin{proof}  Let $\rho:\pi_1(S^4-B)\twoheadrightarrow D_{p}$ be the homomorphism which determines the covering manifold $Y$.
	
	First we compute $g$, the genus of the central surface $F$ of the trisection.  $F$ is a $p$-fold irregular dihedral cover of $S^2$ branched along $2b$ points $\{p_1,\dots,p_{2b}\}$.  Each branch point has one index 1 preimage and $\frac{p-1}{2}$ index 2 preimages. Hence $$\chi(F)=p\chi(S^2-\{p_1,\dots,p_{2b}\})+2b\left(1+\frac{p-1}{2}\right),$$ and the formula for $g$ follows.
	
The 4-manifold $Y_1$ which lies over $X_1=B^4$, where $B\cap X_1$ is $C(\alpha)$, is the cone on the irregular $p$-fold dihedral cover of $S^3$ branched along $\alpha$.  Since $\alpha$ is assumed to be admissible $Y_1\cong B_4$ as well, so $k_1=0$.

Last, we compute $k_i$ for $i=2,3$.  To do this we find the irregular $p$-fold cover of $B^4$ branched along the collection $\mathcal{D}_i$ of $c_i$ properly embedded disks, simultaneously isotopic into $\partial B^4$. A homomorphism $\rho:\pi_1(S^4-B)\twoheadrightarrow D_{p}$ restricts to a homomorphism $\bar{\rho}:\pi_1(B^4-\mathcal{D}_i)\twoheadrightarrow D_{p}$.  Because $\mathcal{D}_i$ is a trivial disk system, we have a homeomorphism of pairs
 $$\phi:(B^4, \mathcal{D}_i)\rightarrow (B^2,\{q_1,\dots,q_{c_i}\})\times I^2,$$
 where $I=[0,1]$.  Let $\iota: (B^2,\{q_1,\dots,q_{c_i}\})\hookrightarrow (B^2,\{q_1,\dots,q_{c_i}\})\times I^2$ denote the inclusion.  The pair $(B^2,\{q_1,\dots,q_{c_i}\})\times I^2$ deformation retracts to $(B^2,\{q_1,\dots,q_{c_i}\})$, so $\iota$ induces an isomorphism on fundamental groups.  Now we have an induced homomorphism $\bar{\rho}\circ \iota_*:\pi_1(B^2-\{q_1,\dots,q_{c_i}\})\twoheadrightarrow D_{p}$.

Let $\Sigma$ denote the irregular $p$-fold dihedral cover of $B^2$ branched along $\{q_1,\dots,q_{c_i}\}$, corresponding to the homomorphism above.  First note that the Euler characteristic of $\Sigma$ is given by
$$\chi(\Sigma)=p(1-c_i)+c_i+c_i\left(\frac{p-1}{2}\right)=2-2g_\Sigma-b_\Sigma,$$
where $g_\Sigma$ and $b_\Sigma$ are the genus and number of boundary components of $\Sigma$ respectively.  Note that $g_\Sigma$ and $b_\Sigma$ depend on the homomorphism $\bar{\rho}\circ \iota_*$, not just on $p$ and $c_i$.

  The product $\Sigma\times I$ is a handlebody, whose boundary is the double of $\Sigma$.  Hence $\Sigma\times I$ homeomorphic to a boundary connected sum of $2g\Sigma+b_\Sigma-1$ solid tori.  The cover of $B^4$ branched along $\mathcal{D}_i$ is homeomorphic to $\Sigma\times I^2$, which is a boundary connect sum of $2g\Sigma+b_\Sigma-1$ copies of $S^1\times B^3$.  Hence $k_i=2g_\Sigma+b_\Sigma-1$.  Using the above formula for $\chi(\Sigma)$, it follows that 
  $$k_i=1-p(1-c_i)-c_i-c_i\left(\frac{p-1}{2}\right).$$
\end{proof}

\section{Dihedral covers $\mathbb{CP}^2\rightarrow S^4$ and other infinite families} 

We give a couple of ways to approach the proof of our Theorem~\ref{manyCP2s}. To determine the homeomorphism types of the covers constructed, one could go quickly using a big hammer~\cite{freedman1982topology}. By our Theorem~\ref{genus-one}, the diffeomorphism types can be determined with the help of the classification of genus one trisections (this bypasses Freedman). 
However, with an eye toward more general constructions and classification results, we complement these very brief arguments by a hands-on procedure to write down trisection diagrams for the covers constructed. This procedure can be used to study more complicated singularities -- and covering manifolds whose intersection forms are more complicated -- where the above considerations no longer suffice.  For instance, we apply the trisection method to show that, by using the family of singularities $\alpha_{6i+3}$, rather than the family $\alpha_{6i}$ used in the Proof of Theorem~\ref{manyCP2s}, we can construct infinitely many 3-fold dihedral covers $f: S^4\to S^4$ branched along {\it non-embedded} two-spheres, each with a link singularity. See Example~\ref{s4tos4}.

\begin{proof}[Proof of Theorem~\ref{manyCP2s}]
Let $f: Y\to S^4$ be as given. By considering the degree of the restriction of $f$ to its unbranched and branched components, we find that \begin{equation}\label{eulerchar.eqn}\chi(Y)=p\chi(S^4)-\dfrac{p-1}{2}\chi(B)-\dfrac{p-1}{2}m.\end{equation}
	
	Now suppose that $Y$ has the homotopy type of $\mathbb{C}P^2$. 
	Plugging into the above equation and simplifying yields 
	 $\frac{p-1}{2}(2+2g-m)=1,$ with $\frac{p-1}{2}$ a positive integer.  Hence, $p=3$ and $g=\dfrac{m-1}{2}$.
	 
	When $m=1$, an explicit family $f_i$ of singularities for such covers can be constructed using the knots $\alpha_k$ in Figure \ref{singularityfamily.fig} as singularities, by setting $k=6i$, $i\in \mathbb{N}$.  The knot $\alpha_{6i}$ is two-bridge, three-colorable, and, by~\cite{lamm2000symmetric}, ribbon.   
	 
	 \begin{figure}[h]\includegraphics[width=4in]{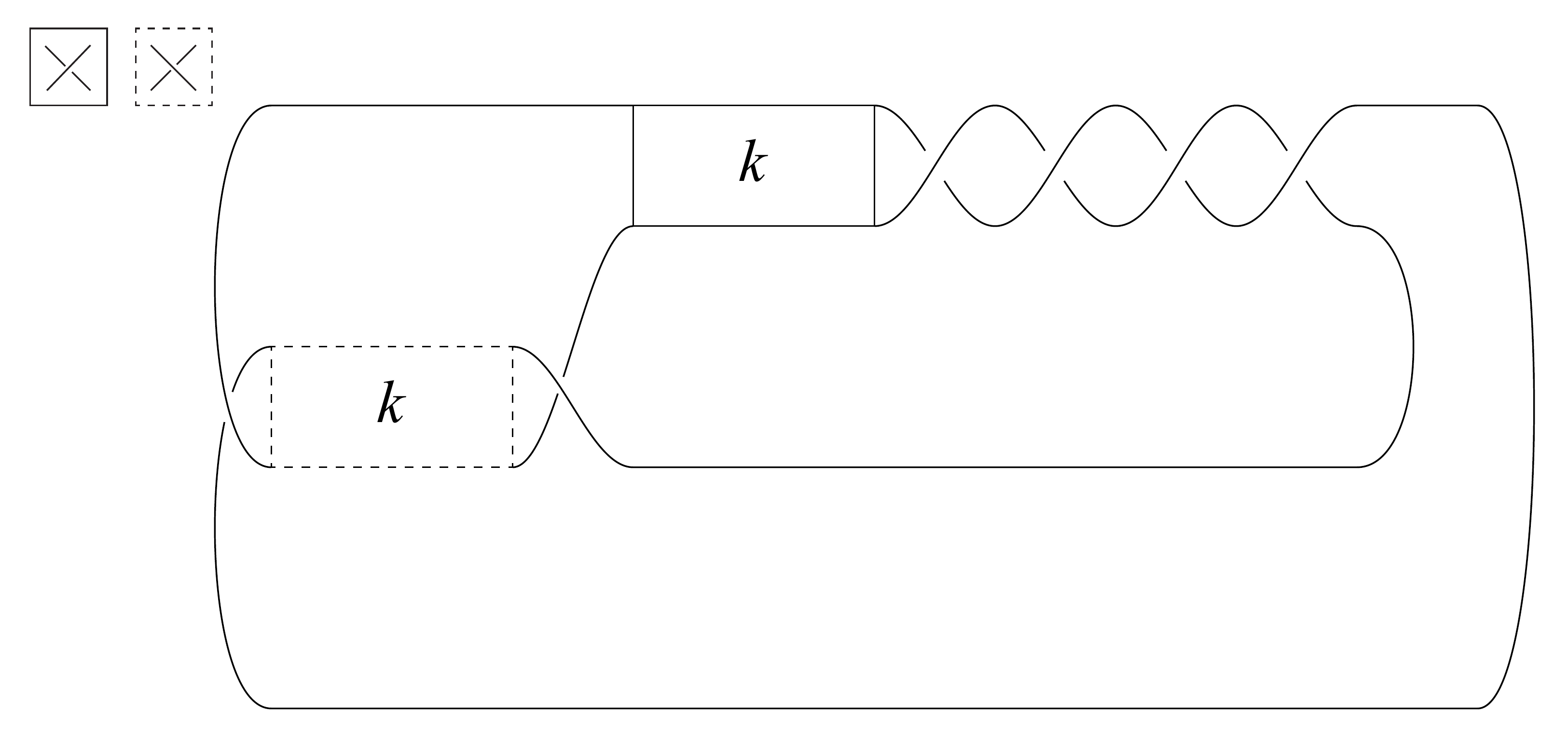}
		 \caption{A family of links $\alpha_k$.  In each of the dotted and solid boxes, there are $k$ copies of the crossings in the corresponding boxes.  $\alpha_k$ has 2 components if $k$ is odd and 1 component otherwise.}
		 \label{singularityfamily.fig}
		 \end{figure}
	
	The covering maps $f_i$ are constructed as follows. Let $\Delta_k$ be a smoothly slice disk for $\alpha_k$, obtained from a ribbon disk pushed into $B^4=\partial S^3$. By Lemma~3.3 of~\cite{kjuchukova2016classification}, the homomorphism $\pi_1(S^3-\alpha)\twoheadrightarrow D_3$ extends to a homomorphism $\pi_1(B^4-\Delta_k)\twoheadrightarrow D_3$, and the corresponding cover is simply connected.  Moreover, since the branching set is the boundary union of the cone on $\alpha_k$ and $\Delta_k$, the Euler characteristic of $Y$ is 3.   
	Since $Y$ is a simply-connected closed oriented four-manifold, it follows that the rank of $H_2(Y;\mathbb{Z})$ is 1.  Hence, $\sigma(Y)=\pm 1$. $Y$ is also smooth, since  $\Delta_k$ is a smoothly embedded disk.  	$Y$ is homeomorphic to $\mathbb{CP}^2$ by~\cite{freedman1982topology}, and diffeomorphic by Theorem~\ref{genus-one}. 
	
\end{proof}

\begin{proof}[Proof of Theorem~\ref{homoCP2}] We apply the argument in the proof of Theorem~\ref{manyCP2s}, with $\alpha_k$ replaced by any 3-admissible homotopically ribbon singularity $\alpha$ and $\Delta_k$ replaced by a homotopically ribbon disk $D$ for $\alpha$. 
We conclude that if $Y$ is a 3-fold dihedral cover of $S^4$, with branching set a boundary union of $D$ and the cone on $\alpha$, then $Y$ is a simply connected four-manifold with $\chi(Y)=3$. Again, it follows that the rank of $H_2(Y;\mathbb{Z})$ is 1 and, by~\cite{freedman1982topology}, $Y$ has the homotopy type of $\mathbb{CP}^2$. If $D$ is also smooth, $Y$ is homeomorphic to $\mathbb{CP}^2$.  
 \end{proof}

\begin{proof}[Proof of Theorem~\ref{ribbon-defect}]
	
	Let $\alpha$ be a knot as in the statement of the theorem. Since $\alpha$ is an admissible singularity type for a $p$-fold cover, $\alpha$ itself is $p$-colorable and, moreover, the irregular dihedral $p$-fold cover of $\alpha$ is $S^3$. Let $D\subset B^4$ be a homotopically ribbon disk for $\alpha$. Fix a $p$-coloring of $\alpha$ and let us consider the corresponding invariant $\Xi_p(\alpha)$. By applying the procedure used in the proof of Theorem~1.5 of~\cite{kjuchukova2016classification}, we can construct a $p$-fold irregular dihedral cover $g: M\to S^4$ such that $M$ is a simply-connected manifold and the branching set of $g$ is a two-sphere $S$ with one singularity of type $\alpha$.  By Equation \ref{eulerchar.eqn} we find $\chi(M)=(p+3)/2$.  Since $M$ is simply connected, it follows that the rank of $H_2(M;\mathbb{Z})$ is $(p-1)/2$.  Hence $|\sigma(M)|\leq (p-1)/2$.   Because the base of $g$ is $S^4$, Equation~2 of Theorem~1.4 of~\cite{kjuchukova2016classification}, $$\sigma(M)=3\sigma(S^4)-\frac{p-1}{4}e(S) + \Xi_p(\alpha),$$ simplifies to $\sigma(M)=\Xi_p(\alpha)$.  Hence $|\Xi_p(\alpha)|\leq (p-1)/2$.
    
  Using Equation~2 of Theorem~1.4 of~\cite{kjuchukova2016classification} again, for any 3-fold dihedral cover $f: Y\to X$ as in the statement of this corollary, we have: 
 $$\sigma(Y)=3\sigma(X)-\frac{p-1}{4} e(B) + \Xi_3(\alpha) = 3\sigma(X)-\frac{p-1}{4} e(B) \pm1.$$ As noted earlier, the sign of the last term changes when $\alpha$ is replaced by its mirror. \end{proof}

\begin{proof}[Proof of Corollary~\ref{twobridge}] Since $\alpha$ is $p$-colorable, two-bridge and slice, $\alpha$ is $p$-admissible~\cite{kjuchukova2016classification}. Since it is two-bridge and slice, by Lisca's result~\cite{lisca2007lens}, $\alpha$ is also ribbon and, in particular, homotopically ribbon. Therefore, by Theorem~\ref{ribbon-defect}, $|\Xi_p(\alpha)| \leq (p-1)/2$, and if $p=3$, $|\Xi_3(\alpha)| = 1$. Equation~\ref{Xi=1} holds by the same argument as in the proof of Theorem~\ref{ribbon-defect}.
 \end{proof}

\subsection{An infinite family of covers via the classification of genus one trisections.} Now we give a description of the infinite family of dihedral covers $\mathbb{C}P^2\rightarrow S^4$ (and, by taking mirror images, $\overline{\mathbb{C}P}^2\rightarrow S^4$), again with singularities $\alpha_{6i}$, using the classification of genus one trisections.  Note that if $p=3$ and the central surface of the trisection has genus one, then by Theorem \ref{trisectionlift.thm}, $b=3$.  We present the branching set $B$ using the tri-plane diagram in Figure \ref{triplanediagramfamily.fig}.    In the solid box in the tangle $B_k$, one inserts $k$ vertically stacked copies of the crossing pictured in the solid box in the upper left, and similarly for the dotted box.  The coloring extends over the new arcs when $k$ is a multiple of $3$.  When $k=6i$, each tri-plane diagram describes a singular $(3;1,2,2)$-bridge trisection with singularity $\alpha_{6i}$.  An Euler characteristic argument shows that $\chi(B)=2$, so this tri-plane diagram also describes a slice disk for $\alpha_{6i}$.  These slice disks are not necessarily related to the disks used in the constructions above.  The singularity for $k=0$ is the knot $6_1$, though this tri-plane diagram differs from the one in Figure \ref{triplanediagram61.fig} by concatenation by a braid, which ensures that the pattern of colors along the bottom of the tangle is $(2,2,3,3,1,1)$.  This makes it easier to construct the cover explicitly.  The following theorem implies that the irregular 3-fold cover of $S^4$ branched along $B$ is diffeomorphic to $\mathbb{C}P^2$ or $\overline{\mathbb{C}P}^2$.

\begin{figure}[hbtp]\includegraphics[width=5in]{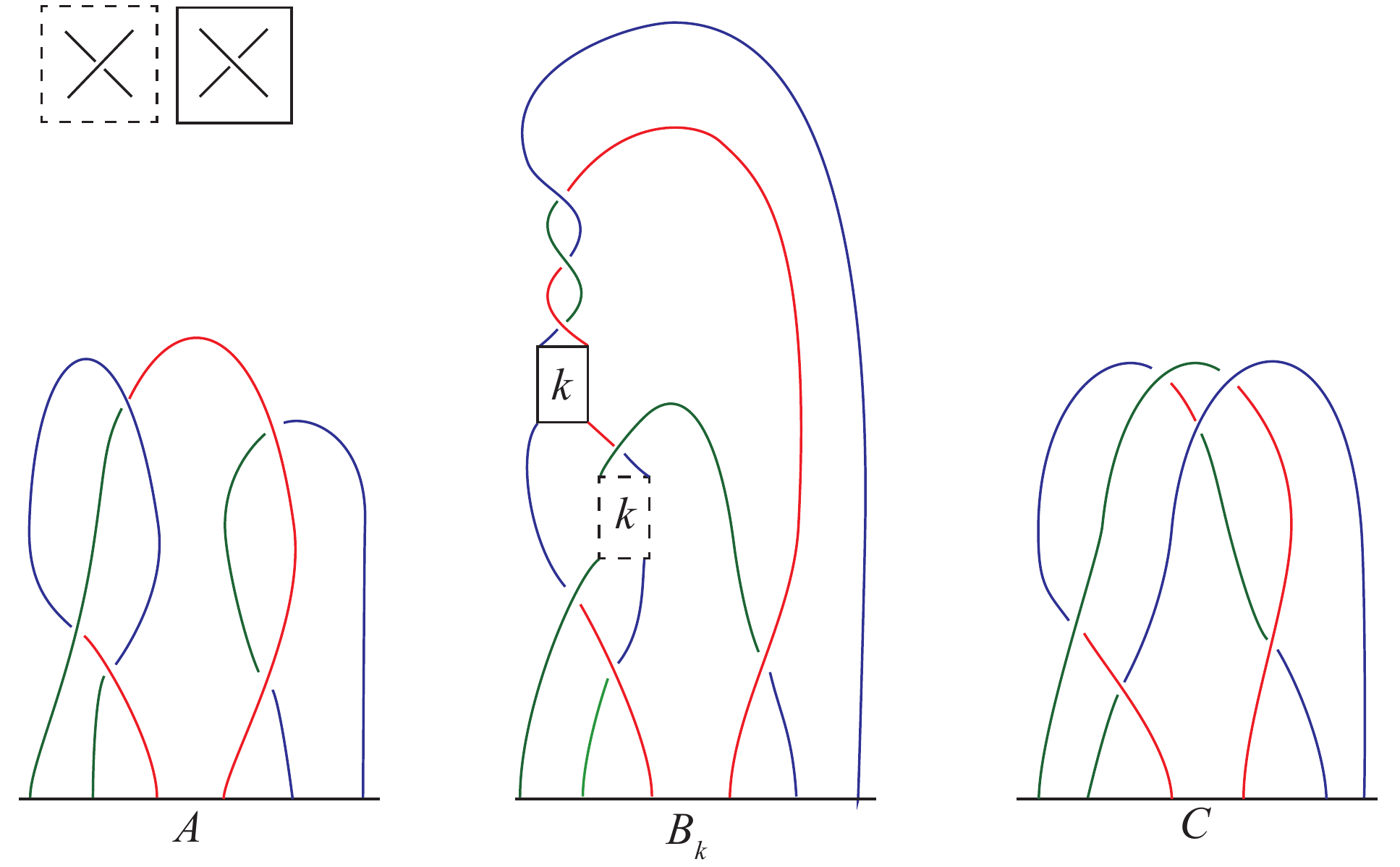}
	\caption{A family of tri-plane diagrams $(A, B_{k}, C)$. When $k$ is a multiple of 3, the coloring is given by extending the coloring in Figure \ref{triplanediagram61.fig} to the new crossings. The colors on the ends of the tangle, from left to right, are $(2,2,3,3,1,1)$.}
	\label{triplanediagramfamily.fig}
	\end{figure}
 \begin{thm} \label{genus-one}
 Let $B\subset S^4$ be a trisected singular surface with one cone singularity of type $\alpha$, and let $\rho:\pi_1(S^4-B)\rightarrow D_3$ be a surjective homomorphism.  Suppose that the three-fold dihedral cover of $S^3$ branched along $\alpha$ is $S^3$.  Then if $b=3$, the tri-plane diagram for $B$ must be a $(3;1,2,2)$ tri-plane diagram, $B$ is homeomorphic to $S^2$, $\alpha$ is slice, and the irregular 3-fold dihedral cover of $S^4$ branched along $B$ is diffeomorphic to $\mathbb{C}P^2$ or $\overline{\mathbb{C}P}^2$.
 	
	 	\end{thm}

 \begin{proof} Suppose that $B$ has a 3-colored tri-plane diagram of type $(3;1,c_2,c_3)$.  Then $c_2$ and $c_3$ are each greater than or equal to two.
 	 
	  	 On the other hand, $\chi(B)=1+c_2+c_3-3\leq 2$, so $c_2+c_3\leq 4$. Hence $c_2=c_3=2$.  This implies $B$ is $S^2$ and $\alpha$ is slice.
	
 	Now we consider the corresponding trisection diagram of the irregular 3-fold dihedral cover of $S^4$ branched along $Y$.  Since $b=3$, we must have $g=1$.  Furthermore, if the tri-plane diagram is $(A,B,C)$, both $B\cup \bar{C}$ and $B\cup\bar{A}$ are two-component unlinks $L_1$ and $L_2$ in $S^3$, and $\rho|_{S^3-L_i}$ is surjective.   
 	Thus the irregular 3-fold cover of $S^3$ branched along either one of the $L_i$ must be homeomorphic to $S^3$, since each of the $L_i$ are clearly two-bridge links.  Thus we get a $(1;0,0,0)$ trisection.   Hence we see~\cite{meier2016classification} that the cover is $\mathbb{C}P^2$ or $\overline{\mathbb{C}P}^2$.
 	\end{proof}
 
\subsection{Explicit construction of an infinite family of covers via trisection diagrams} Finally, we describe a hands-on method for constructing the same family of 3-fold irregular dihedral covering maps $\mathbb{C}P^2\rightarrow S^4$ using trisections.  This method is general, so it can be used to identify a covering manifold when the simple arguments above do not apply.  By a small modification of our branching set, we also obtain an infinite family of covers $S^4\rightarrow S^4$.

By Theorem \ref{trisectionlift.thm},  the cover $Y_{k}$ corresponding the singularity $\alpha_k$, with $k=6i$, will be equipped with a $(1;0,0,0)$-trisection.  Namely, the central surface $F$ is a torus, and each $Y_k$ is a 4-ball.  The boundaries $\partial Y_k=S^3$ are each decomposed as a union of two solid tori, with Heegaard surface $F$.  Now we examine in more detail how the pieces are built, in order to produce a trisection diagram on the torus for each $Y_{k}$.

	Let $\mathcal{T}=\{t_1,t_2,t_3\}$ be a $3$-strand trivial tangle in $B^3$, with shadows $\{d_1,d_2, d_3\}$ and bridge disks $\{\Delta_1,\Delta_2,\Delta_3\}$, and let $\rho:\pi_1(B^3-\mathcal{T})\twoheadrightarrow D_3$ be a surjective homomorphism.  Then the corresponding irregular 3-fold dihedral cover of $B^3$ branched along $\mathcal{T}$ is a solid torus $T_{\mathcal{T}}$.  For each $i\in \{1,2,3\}$ denote the three lifts of $d_i$ and its corresponding bridge disk by $d_i^j$ and $\Delta_i^j$, for $j=1,2,3$.  For each $i=1,2,3$, there exist $j,k\in\{1,2,3\}$ such that $\Delta^j_i$ and $\Delta^k_i$ share boundary along the index two lift $t^2_i$ of $t_i$.  In particular $\Delta^j_i\cup_{t^2_i}\Delta^k_i$ is a properly embedded disk in $T_{\mathcal{T}}$, and there exist two indices $k,l\in \{1,2,3\}$ such that $d^k_i(0)=d^l_i(0)$ and $d^k_i(1)=d^l_i(1)$.  Hence two of the lifts of each shadow, when concatenated, form a closed loop bounding $\Delta^j_i\cup_{t^2_i}\Delta^k_i$, though we will see shortly this may not be a compressing disk.  We call this loop a {\it closed shadow}.  In general, for a 3-fold dihedral cover, an $n$-strand trivial tangle gives rise to $n$ closed shadows, one for each strand of the tangle.
 	\begin{prop}  \label{meridians.prop}  Given a $3$-strand trivial tangle and a surjective homomorphism $\rho:\pi_1(B^3-\mathcal{T})\twoheadrightarrow D_3$, either two or three of its closed shadows are meridians of the solid torus solid torus $T_{\mathcal{T}}$.
		\end{prop}

\begin{proof}  Let $\mu_i$ be a meridian of $t_i$ in $B^3$.  We view $(B^3,\{t_1,t_2,t_3\})=(D^2,\{p_1,p_2,p_3\})\times [0,1]$.  The cover of $D^2$ branched along $\{p_1,p_2,p_3\}$ is an annulus.   
	
	First we consider the case where the values $\rho(\mu_i)$ are all distinct.  Without loss of generality, assume $\rho(\mu_i)=i\in \{1,2,3\}$, where as usual, $i$ corresponds to the transposition with fixed point $i$.  An explicit construction of the cover is given in the top row of Figure \ref{annuluslifts.fig}, and it is clear that each closed shadow is a meridian.
	
	The remaining case is that two exactly values of $\rho(\mu_i)$ are equal.  Without loss of generality, assume $\rho(\mu_1)=\rho(\mu_2)=1$, and $\rho(\mu_3)=2$.  An explicit construction of the cover is given in the bottom row of Figure \ref{annuluslifts.fig}.  In this case, two closed shadows are meridians; the other is a nullhomotopic curve on $\partial T_{\mathcal{T}}$. \end{proof}
	
	\begin{figure}[htbp]\includegraphics[width=3in]{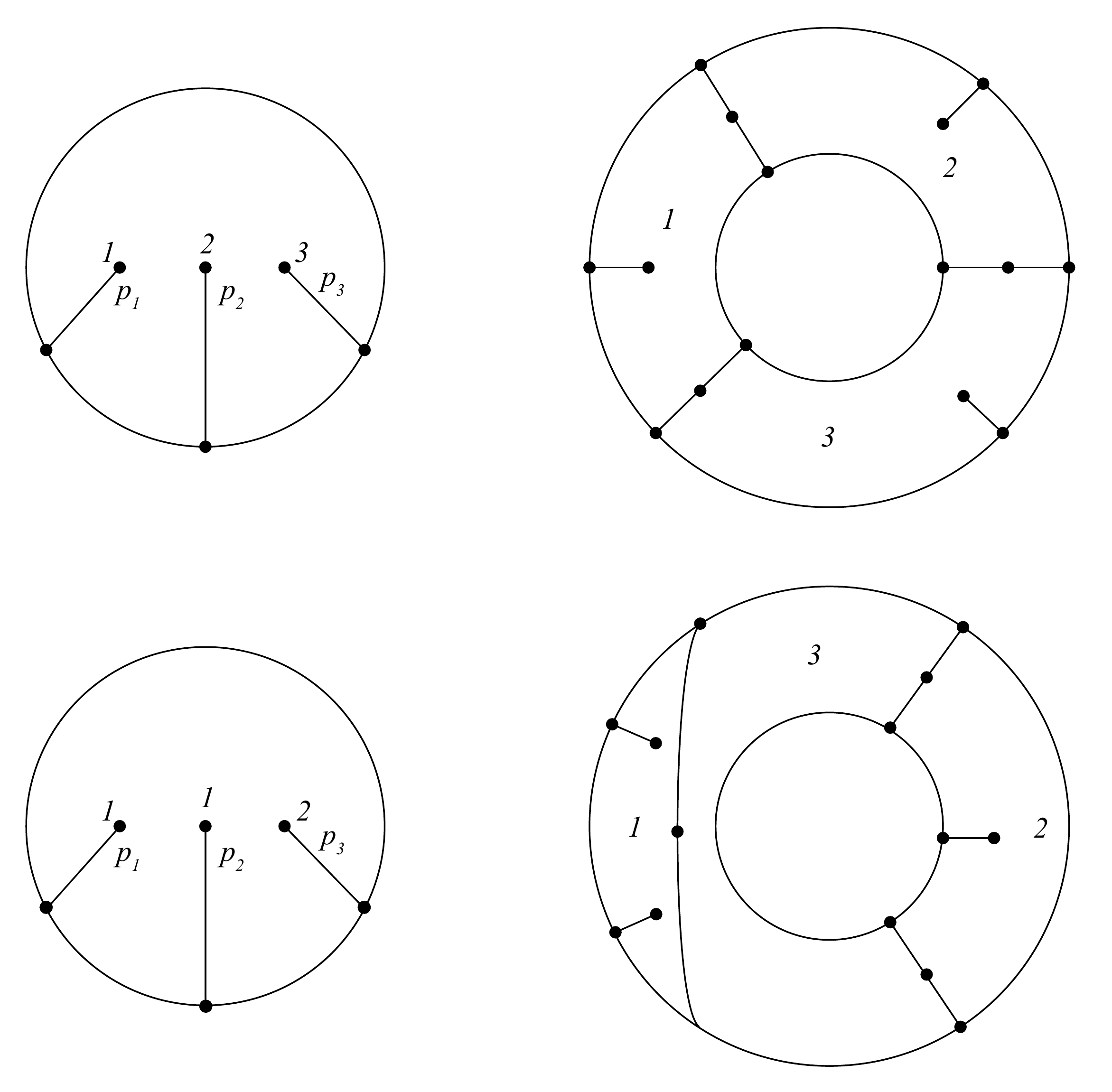}
		\caption{Two 3-fold irregular dihedral covers of the disk branched over three points.}
		\label{annuluslifts.fig}
		\end{figure}
	
\begin{ex}		We begin by constructing the 3-fold irregular cover of $S^4$ branched along a two-sphere with a singularity of type $\alpha=6_1$, where the two-sphere is presented by the tri-plane diagram $(A,B_0,C)$ in Figure \ref{triplanediagramfamily.fig}.  Each tangle is trivial, so for each tangle, we may choose three shadows. These paths are pictured in Figure \ref{diskbottoms61.fig},
			\begin{figure}[h]\includegraphics[width=3in]{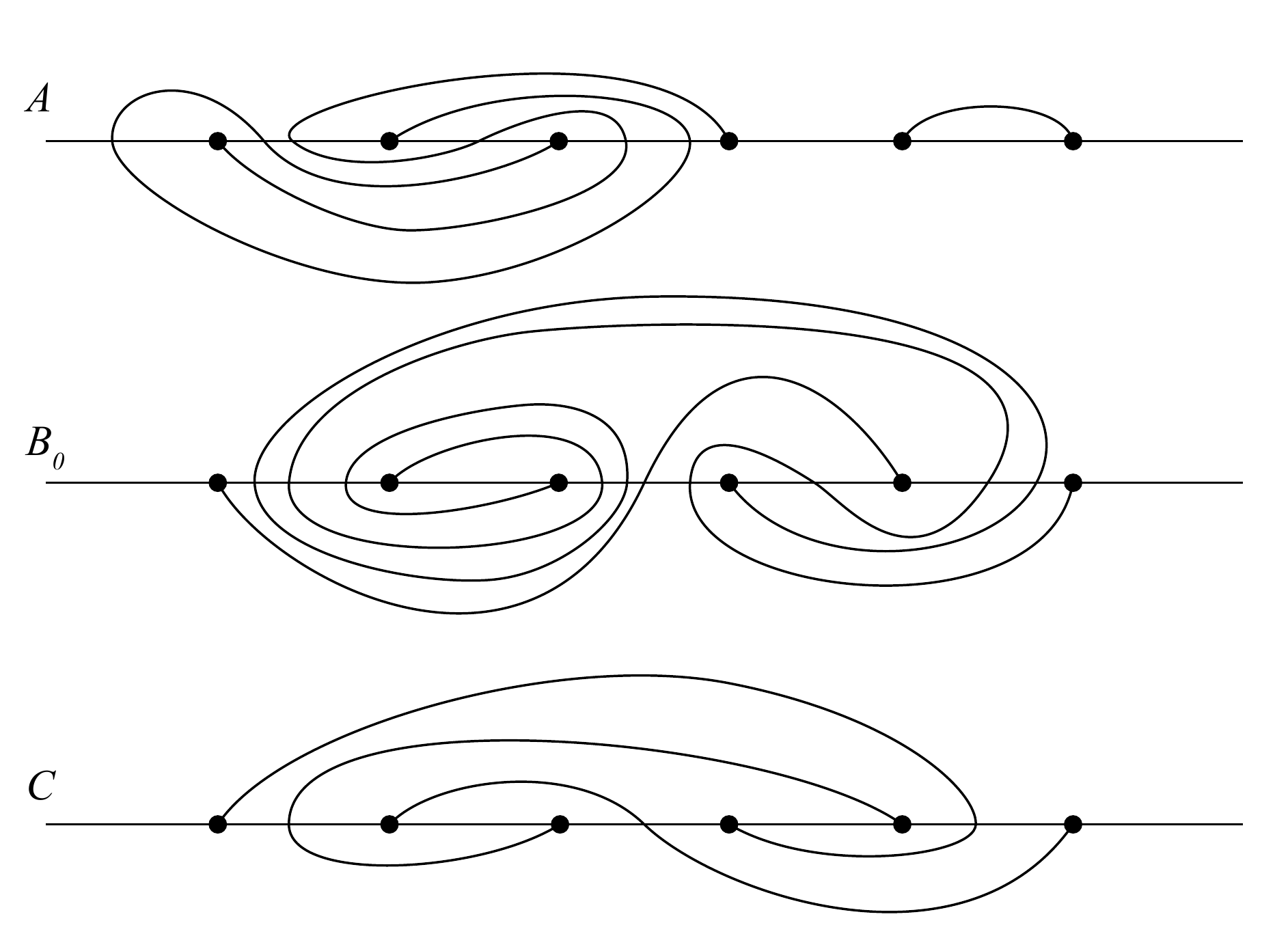}
			\caption{Shadows for the trivial tangles $A$, $B_0$, and $C$ in Figure \ref{triplanediagramfamily.fig}.}
			\label{diskbottoms61.fig}
			\end{figure}
By Proposition ~\ref{meridians.prop}, for each tangle $A$, $B_0$, and $C$, either two or three of the corresponding closed shadows are meridians in the covering solid torus.  It turns out that for $A$ and $B_0$, two are meridians, while for $C$, all are meridians.  For each tangle $A$, $B_0$, and $C$, we choose a shadow such that the corresponding closed shadow is a meridian.  These three shadows are drawn on the sphere in Figure \ref{spheretoruscover.fig}.  We equip $S^2$ with a cell structure, consisting of vertices $\{a,b,c,d,e,f\}$, edges $x_i$ and $y_i$, and two 2-cells $D$ and $E$.  This cell structure lifts to a cell structure on the torus, and allows us to draw the closed shadows corresponding to the shadow arcs on $S^2$.   The result is a trisection diagram on the torus.  For $(A,B_0,C)$ we obtain a trisection diagram for $\mathbb{C}P^2$.
\end{ex}

			\begin{figure}[htbp]\includegraphics[width=6in]{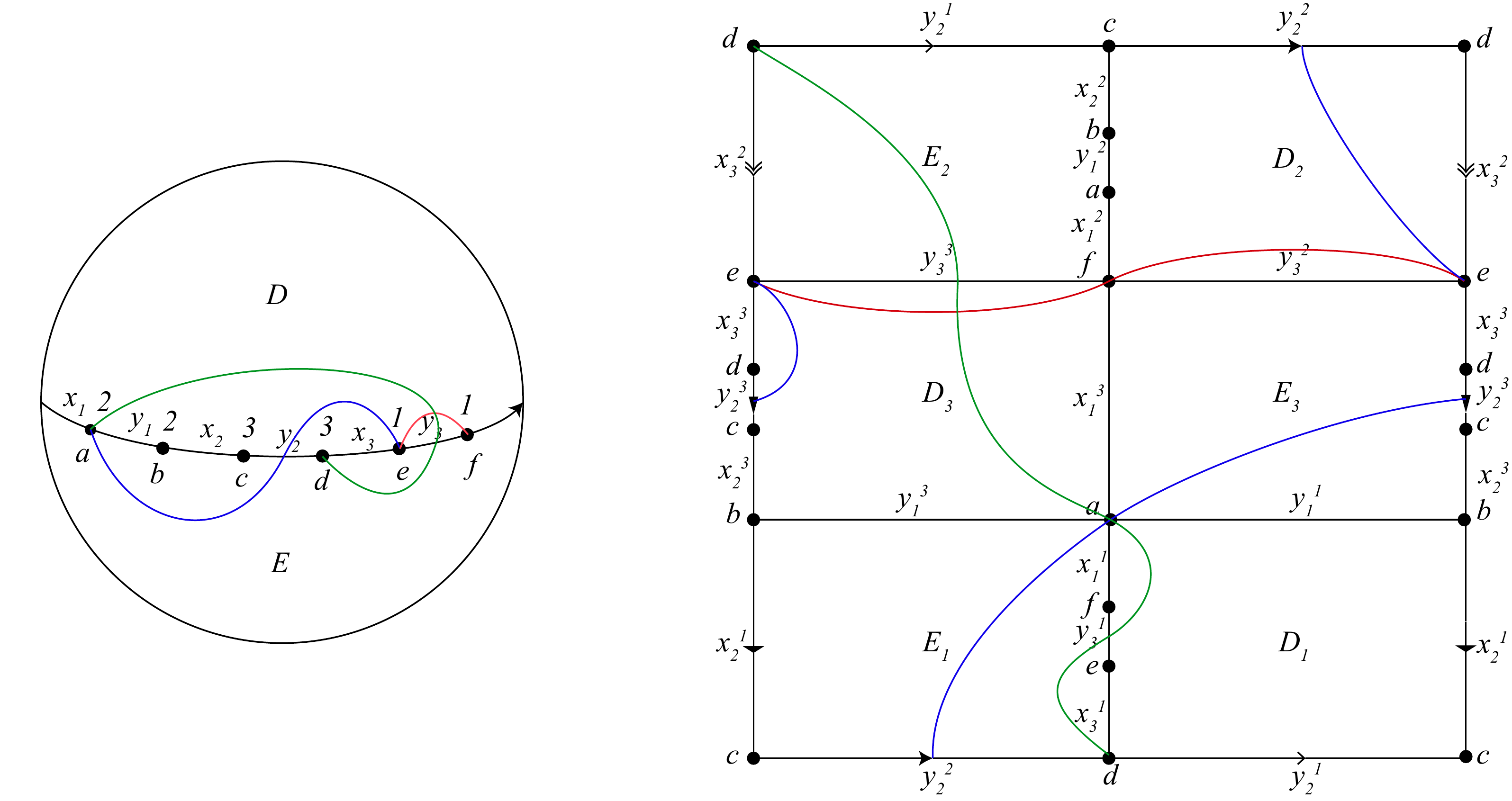}
				\caption{Lifts of one shadow for each of the tangles $A$ (red), $B_0$ (blue), and $C$ (green) in Figure \ref{triplanediagramfamily.fig} to the torus.  Note that the gluing of the top and bottom edges is not the standard one.}
				\label{spheretoruscover.fig}
				\end{figure}

	Now we introduce notation for the lifts of our cells to the torus, in order to lift the more complicated shadows for the family of tangles $(A,B_k,C)$.  We cut along the paths $y_1$, $y_2$, and $y_3$ in $S^2$ to obtain a sphere with three holes, and take three copies of the result.  Call these $P_1$, $P_2$, and $P_3$.  The cells on $P_i$ are labeled in Figure \ref{torusassembly.fig}.  To obtain the torus in Figure \ref{spheretoruscover.fig}, we make the following identifications, according to the colors on the branch points:
	
	$$y_1^1\sim z_1^3,\quad y_1^2\sim z_1^2, \quad y_1^3\sim z_1^1$$
	$$y_2^1\sim z_2^2,\quad y_2^2\sim z_2^1, \quad y_2^3\sim z_2^3$$
	$$y_3^1\sim z_3^1, \quad y_3^2\sim z_3^3, \quad y_3^3 \sim z_3^2$$	
				\begin{figure}\includegraphics[width=2in]{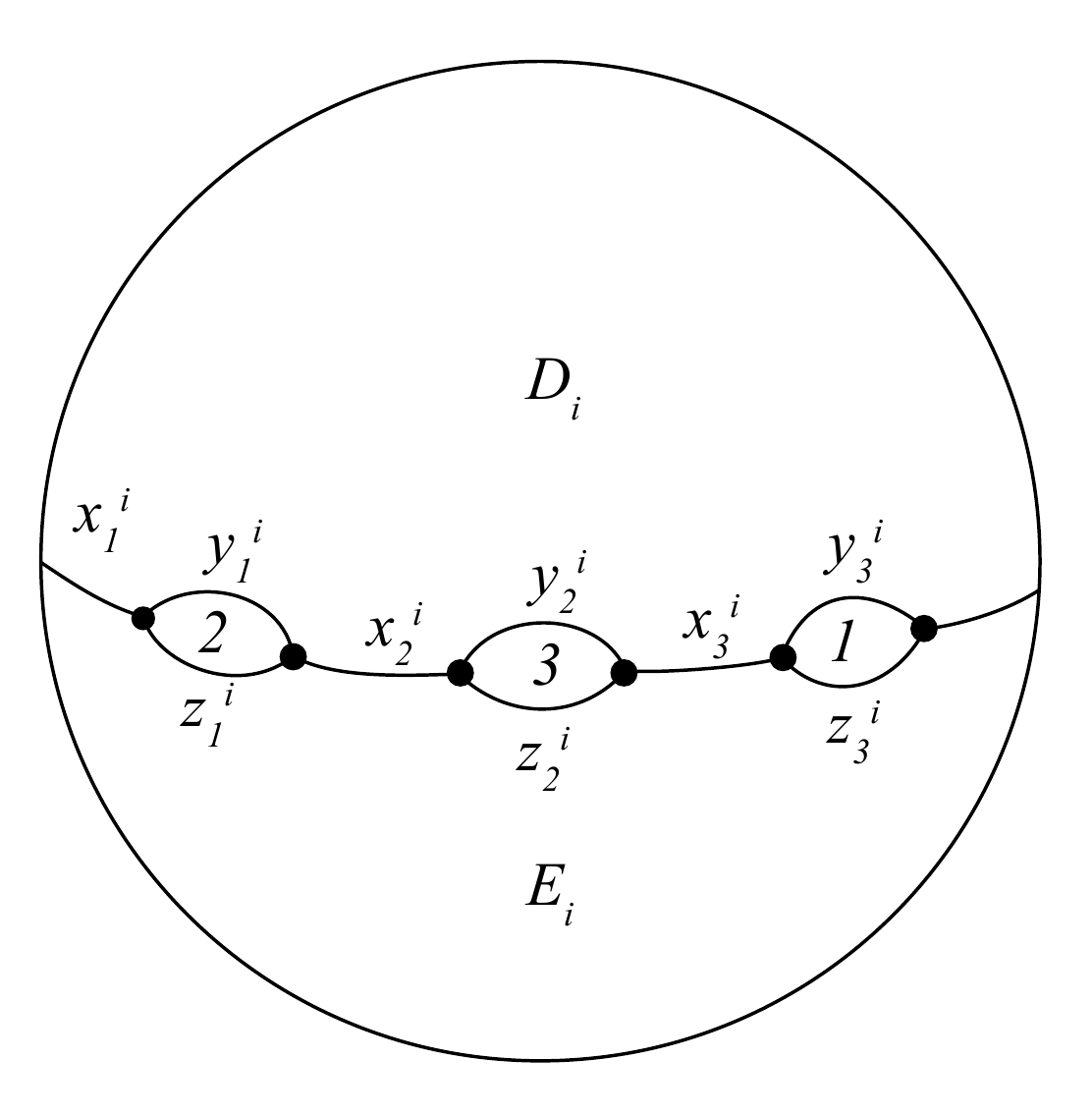}
					\caption{Notation for the three lifts of the cells in the structure on $S^2$.}
					\label{torusassembly.fig}
					\end{figure}
Next we draw a shadow for each of the $B_k$.  Note that our tri-plane diagram is only 3-colorable when $k$ is a multiple of 3, and $\alpha_k$ is only a knot when $k$ is even, but it is easier to describe the shadows for all $k$. The shadows for $B_0$ and $B_1$ are pictured in the first two lines of Figure \ref{diskfamily.fig}. The shadows for $B_{2k}$ and $B_{2k+1}$ respectively are obtained from the boundary paths for $B_0$ and $B_1$ by applying successive twists along the dotted curve.
					
\begin{figure}\includegraphics[width=3in]{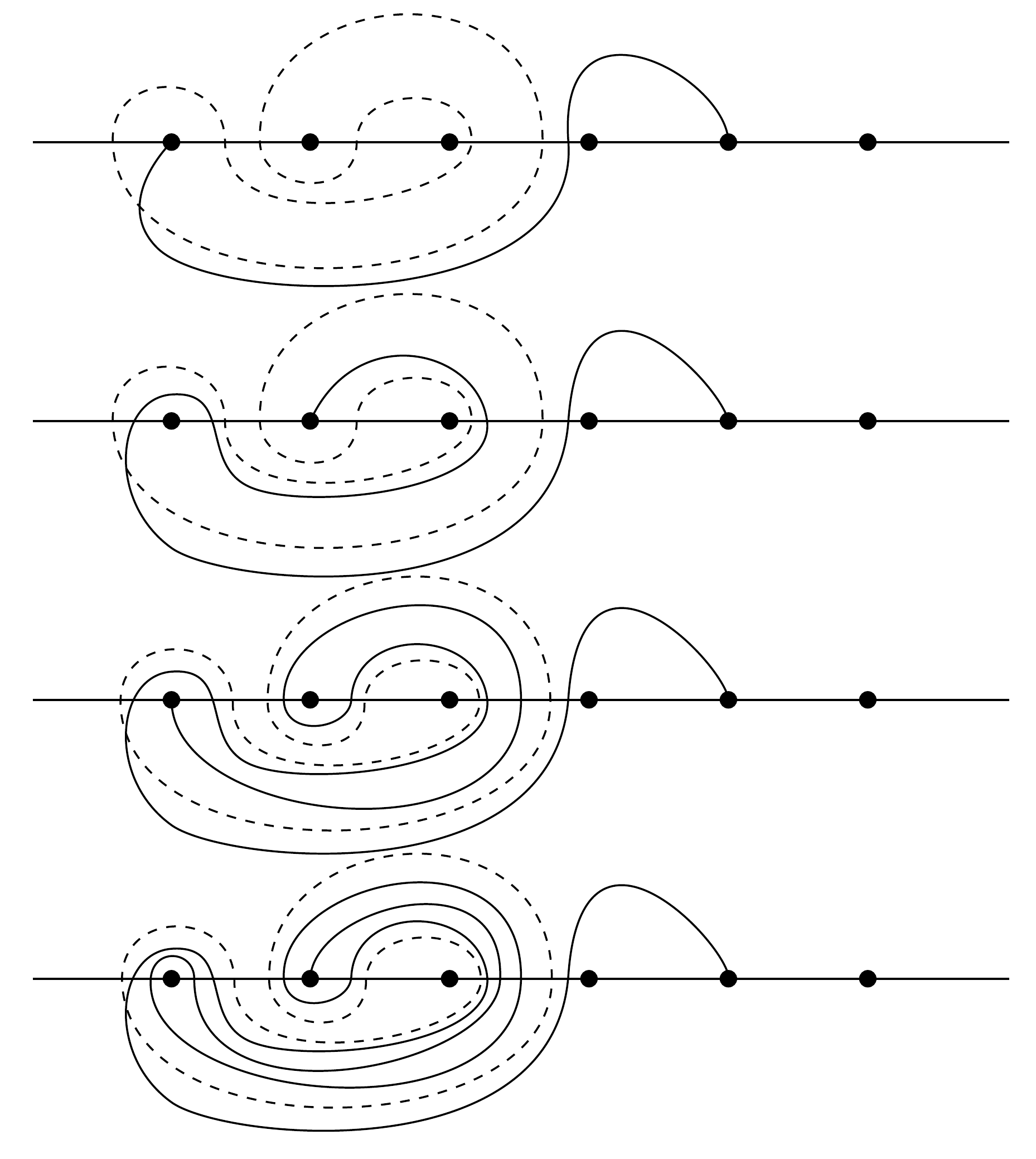}
	\caption{Shadows for the families $B_{2k}$ and $B_{2k+1}$.}
	\label{diskfamily.fig}
\end{figure}

Each shadow can be represented as a word in the $x_i$ and $y_i$.  Its lift to the torus can be represented as a word in the $x_i^j$ and $y_i^j$.

\begin{ex}\label{s4tos4}
We construct the cover of $S^4$ branched along the colored surface described by $(A,B_{6i},C)$. The shadow of $B_{6i}$ is represented by the word $$y_2(x_1y_1y_2x_2y_1y_2)^{3i}.$$  The two lifts of this path to the torus, which together form a closed shadow, are
$$y_2^2\left((x_1^1y_1^1y_2^3x_2^3y_1^1y_2^1)(x_1^2y_1^2y_2^1x_2^1y_1^3y_2^3)(x_1^3y_1^3y_2^2x_2^2y_1^2y_2^2)\right)^i$$ and 
$$y_2^3\left((x_1^3y_1^3y_2^2x_2^2y_1^2y_2^2)(x_1^1y_1^1y_2^3x_2^3y_1^1y_2^1)(x_1^2y_1^2y_2^1x_2^1y_1^3y_2^3)\right)^i.$$

The corresponding closed shadows are all homotopic to the closed shadow for $B_0$.  Hence in all cases the cover is $\mathbb{C}P^2$ or $\overline{\mathbb{C}P}^2$, depending on choice of orientation.
\end{ex}
\begin{ex}
We construct the cover of $S^4$ branched along the colored surface described by $(A,B_{6i+3},C)$.   The shadow of $B_{6i+3}$ is represented by the word $$y_2(x_1y_1y_2x_2y_1y_2)^{3k}(x_1y_1y_2x_2y_1y_2x_1y_1y_2).$$  The two lifts of this path to the torus, which together form a closed shadow, are represented by the words
$$y_2^2\left((x_1^1y_1^1y_2^3x_2^3y_1^1y_2^1)(x_1^2y_1^2y_2^1x_2^1y_1^3y_2^3)(x_1^3y_1^3y_2^2x_2^2y_1^2y_2^2)\right)^i(x_1^1y_1^1y_2^3x_2^3y_1^1y_2^1x_1^2y_1^2y_2^1)$$
and 
$$y_2^3\left((x_1^3y_1^3y_2^2x_2^2y_1^2y_2^2)(x_1^1y_1^1y_2^3x_2^3y_1^1y_2^1)(x_1^2y_1^2y_2^1x_2^1y_1^3y_2^3)\right)^i(x_1^3y_1^3y_2^2x_2^2y_1^2y_2^2x_1^1y_1^1y_2^3).$$

In this case the closed shadows are isotopic to those of the tangle $C$, so the corresponding cover is $S^4$~\cite{meier2016classification}.  In this case the singularity $\alpha_{6i+3}$ is a two-component link, so the branching set has a self-intersection.

\end{ex}

\section{Computing the Signature Defect}\label{examples}
So far we have seen only knots whose signature defect is $\pm 1$.  Finding colored tri-plane diagrams for a given singularity can be difficult, especially when $\alpha$ has large dihedral 4-genus.   Here we present a combinatorial procedure for computing the defect from a knot diagram, which can be used to determine the defect for any admissible knot.   
\subsection{Overview of the procedure }

The formula for the signature defect involves invariants of $\beta$, a mod 3 characteristic knot for $\alpha$, as well as the signature of a matrix whose entries are linking numbers of curves in an irregular $p$-fold dihedral cover of $S^3$ branched along $\alpha$.  These linking numbers represent intersection numbers of relative cycles in a four-manifold $W_{\alpha,\beta}$, a cobordism from the irregular $p$-fold dihedral cover of $S^3$ along $\alpha$ to the cyclic cover of $S^3$ along $\beta$, constructed by Cappell and Shaneson \cite{CS1984linking}.

First we briefly explain how the signature defect for $\alpha$ is computed, and in particular, what work is needed to pass from the geometric formula in \cite{kjuchukova2016classification} to a computation involving only diagrammatic information.

	To compute the defect, we must compute the signature of a matrix of linking numbers of curves which cobound relative cycles in $W_{\alpha,\beta}$.  The first steps are as follows.
	\begin{enumerate}
		\item Choose a diagram and a Seifert surface $V$ for $\alpha$
		\item Find a characteristic knot $\beta\subset V$ for $\alpha$, and choose an orientation for $\beta$
		\item Choose a basis $\mathcal{B}_V=\{\omega_i\}_{i=1}^{2g-2}\cup \{\beta_r,\beta_l\}$for $H_1(V-\beta;\mathbb{Z})$, where $g$ is the genus of $V$ and $\beta_r$ and $\beta_l$ are right and left push-offs of $\beta$ in $V$
				\item Compute the linking numbers of these basis elements in the 3-fold dihedral cover of $S^3$ branched along $\alpha$, using the algorithm in \cite{cahnkjuchukova2016linking}
		\end{enumerate}

The difficulty is that not all of the linking numbers computed above contribute to the defect.  The next step of the procedure is to identify the curves whose linking numbers do appear using only diagrammatic information.  Briefly, we use a construction of the irregular dihedral cover $M_{\alpha}$ of $S^3$ branched along $\alpha$, due to Cappell and Shaneson \cite{CS1984linking}.  In this construction, one begins with the cyclic cover of $S^3$ branched along $\beta$, removes a certain handlebody from the interior to obtain a 3-manifold with one boundary component, and then obtains a closed 3-manifold by gluing that boundary component to itself via an involution.  The resulting three-manifold is $M_\alpha$.  The curves whose linking numbers appear in the defect lie on the boundary of the handlebody above, and must be in the kernel of a map which we discuss in detail in the proof of Theorem \ref{procedurethm}.  Before beginning the proof, we illustrate Theorem \ref{procedurethm} in two examples.

\begin{ex}\label{61procedure.ex} In this example we compute $|\Xi_3(6_1)|$ using Theorem \ref{procedurethm}.

The knot $6_1$ is the three-colorable two-bridge slice knot of smallest crossing number, so is the simplest example to which Corollary \ref{twobridge} applies.  We varify that $|\Xi_3(6_1)|=\pm 1$ independently using Theorem \ref{procedurethm}. We will use the three-coloring and the Seifert surface $V$ pictured in Figure \ref{61seifert.fig}.  
\begin{figure}[htbp]
	\includegraphics[width=2in]{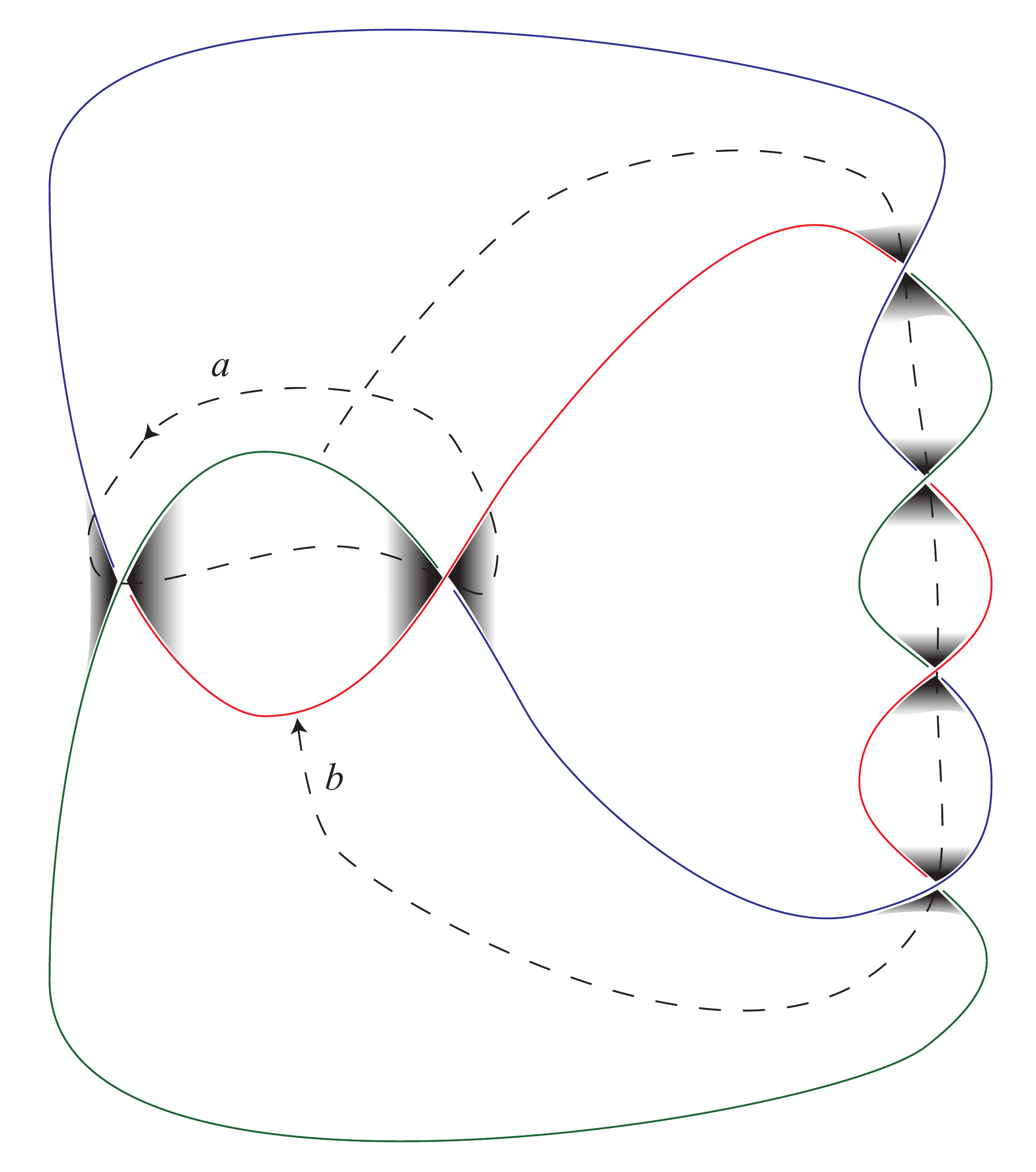}
	\caption{The knot 6-1, a Seifert surface $V$, and a basis $\{a,b\}$ for $H_1(V;\mathbb{Z})$.}\label{61seifert.fig}
	\end{figure}
	\begin{figure}[htbp]
		\includegraphics[width=5in]{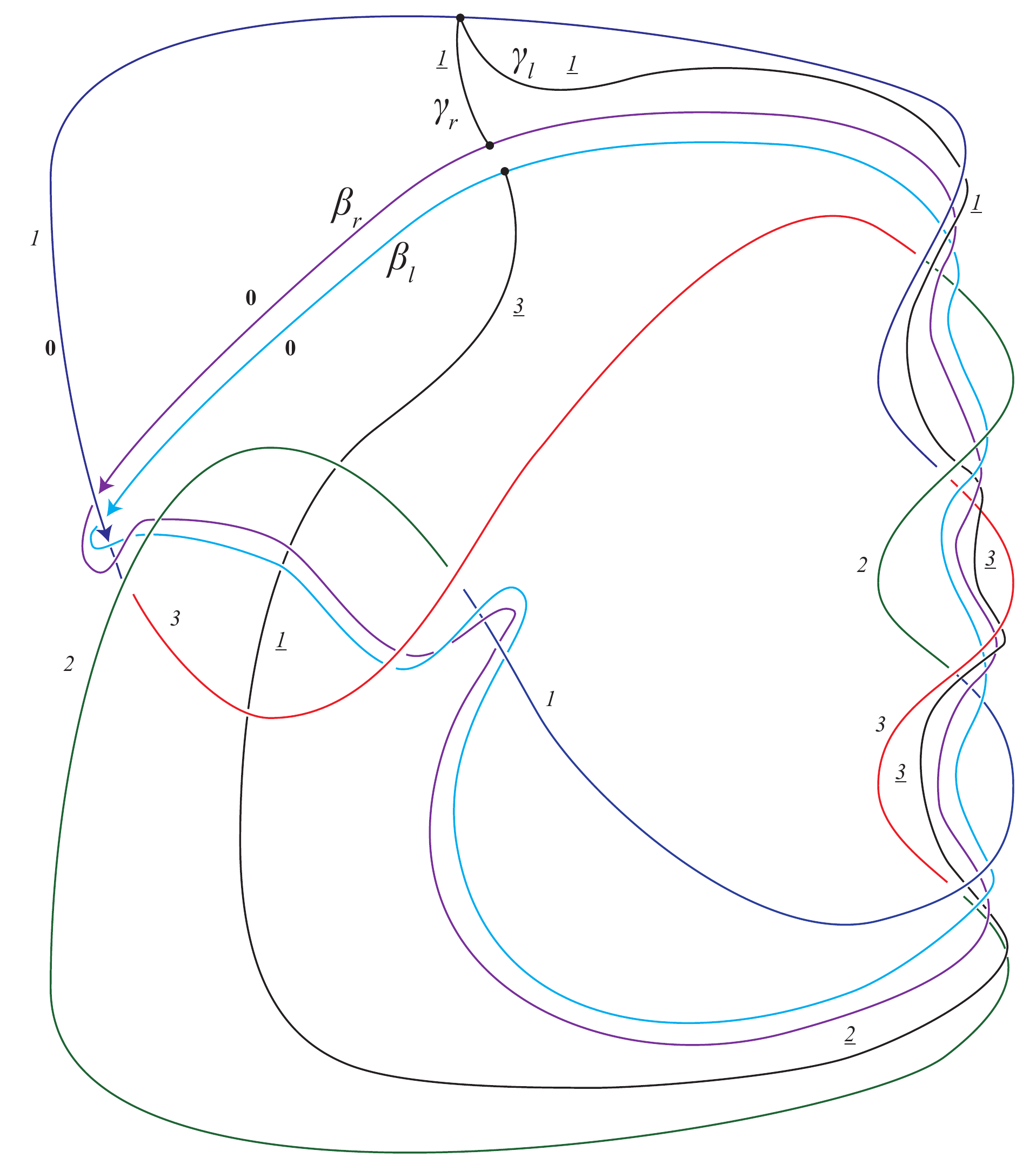}
		\caption{The numbered link diagram for $6_1$, together with the push-offs $\beta_r$ and $\beta_l$ of the characteristic knot, and the anchor paths $\gamma_r$ and $\gamma_l$.}
		\label{6-1numbereddiagram.fig}
		\end{figure}
We begin by finding a mod 3 characteristic knot $\beta$ for this three-colored 6-1 diagram.  With respect to the basis $\{a,b\}$ we compute the symmetrized linking form
 $$L_V+L_V^{t}=\begin{pmatrix}-2&1\\1&4\end{pmatrix}.$$  
 
 Recall that a characteristic knot $\beta$ satisfies $(L_V+L_V^t)\beta \equiv 0 \mod 3$.  Hence $\beta=a-b$ is a mod 3 characteristic knot.  Since $V$ has genus one, our basis $\mathcal{B}_V$ consists only of $\beta_r$ and $\beta_l$.  An embedded representative of the class $\beta$, together with a choice of anchor paths $\gamma_r$ and $\gamma_l$, is shown in Figure \ref{61seifertwithbeta.fig}.   We indicate a numbering of the arcs of $\alpha$, $\beta_r$, and $\beta_l$ by marking the zeroth arc of each in bold.  The other numbers are assigned as described above, but are omitted from Figure \ref{61seifertwithbeta.fig} to avoid clutter.
 \begin{figure}
	\includegraphics[width=2in]{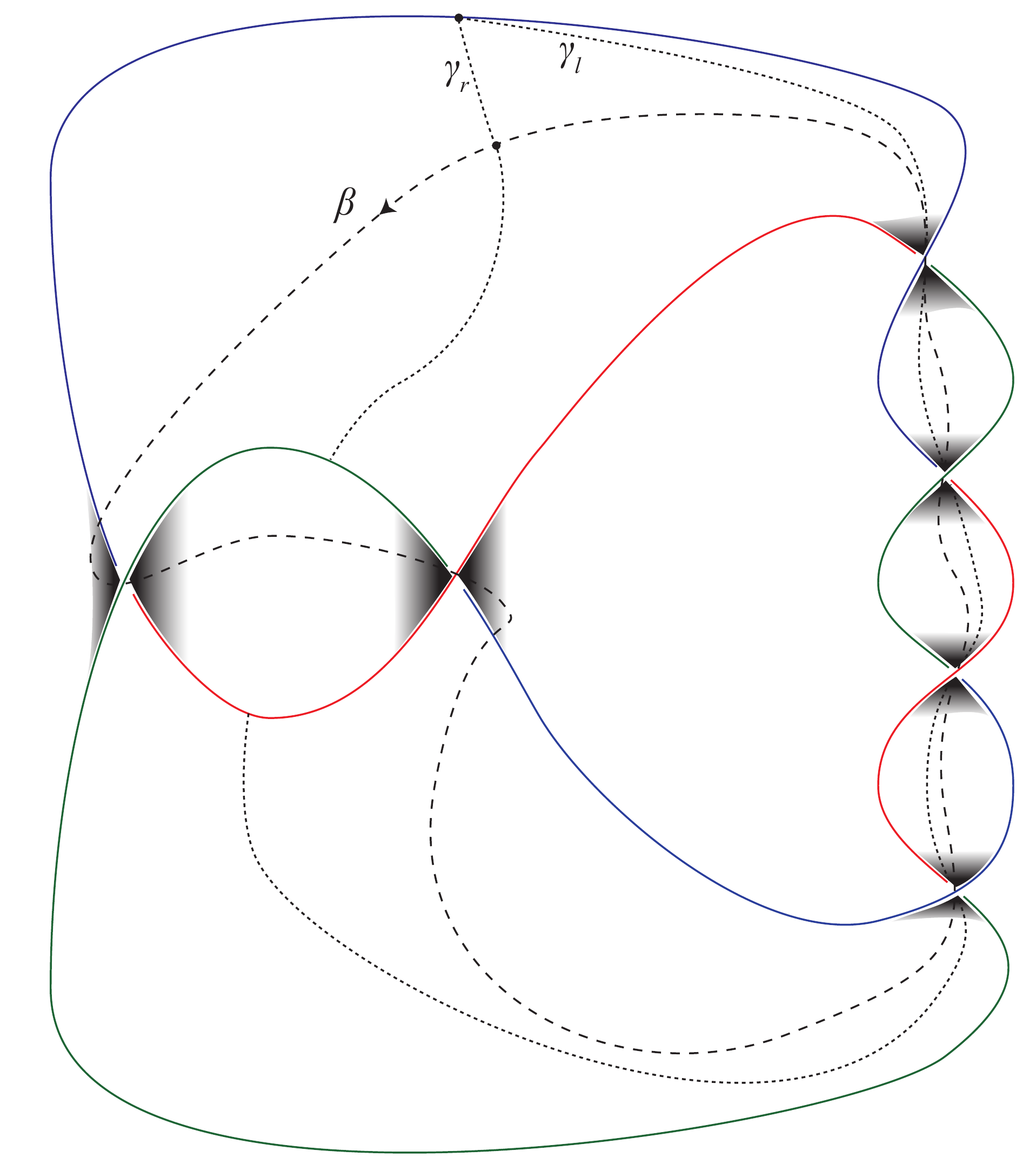}
	\caption{The knot 6-1 together with a characteristic knot $\beta$, and the corresponding arcs $\gamma_r$ and $\gamma_l$.}\label{61seifertwithbeta.fig}
	\end{figure}

The input for the computer program consists of seven lists.  We summarize this briefly here; for detailed examples see \cite{cahnkjuchukova2016linking}.  The first four are associated to the knot $\alpha$.  The remaining six lists are associated to the two curves $g$ and $h$ described in the introduction.  The first list denotes the number $f=(f(i))_i$ on the over-arc which meets the head of arc $i$ of $\alpha$.  The second $\epsilon=(\epsilon(i))_i$ denotes the local writhe number at the head of arc $i$.  The third $t=(t(i))_i$ denotes the {\it type} of crossing at the head of arc $i$; we let $t(i)=p$ of the over-arc at the head of arc $i$ is an arc of $g$, and we let $t(i)=k$ if the over-arc at the head of arc $i$ is another arc of $\alpha$.  Recall that the $i^{th}$ arc of $\alpha$ may be a union of smaller arcs, separated by over-crossings by arcs of $h$, and the over-crossing at the end of an arc of $\alpha$ will never be an arc of $h$ due to our numbering system.  Finally, the fourth list $c=(c(i))_i$ is the color on the $i^{th}$ arc of $\alpha$.

Numbering, signs, crossing types, and colors for $\alpha$:
$$f=(1,8,0,7,10,5,3,2,4,6,6,4)$$
$$\epsilon=(-,+,-,-,-,-,+,+,+,-,+,-)$$
$$t=(p,k,k,p,k,p,p,k,p,k,p,k)$$
$$c=(1,1,3,2,2,1,1,1,2,2,3,3)$$

The remaining lists are the over-crossing numbers, signs, and crossing types for the other two components $g$ and $h$ of the link diagram.

Numbering, signs, and crossing types for $\beta$:
$$(0,8,2,6,6,10,4,0)$$
$$(-,+,-,+,-,+,-,+)$$
$$(k,k,k,k,k,k,k,k)$$

Numbering, signs, and crossing types for $\beta_r$:
$$(0,8,2,3,6,4,6,10,6,4,0)$$
$$(-,+,-,-,+,+,-,+,+,-,+)$$
$$(k,k,k,p,k,p,k,k,p,k,k)$$

The computer program returns the linking numbers of the lifts $\beta_r^j$ and $\beta_l^k$, $j,k=1,2,3$.  They are given by the matrix $$\begin{pmatrix}0&0&1\\0&1&0\\1&0&0\end{pmatrix}.$$  

Next we compute the monodromies of $\gamma_r$ and $\gamma_l$:
$$\mu_{\gamma_r}=Id$$
$$\mu_{\gamma_l}=(23)(13)(12)(23)(12)(13)=(123)$$

The zeroth arc of $\alpha$ is colored $c_0=1$.  Hence $\mu_{\gamma_r}(c_0)=1$ and $\mu_{\gamma_l}(c_0)=3$.

By Theorem~\ref{procedurethm} the signature defect of $6_1$ is the signature of the 1 by 1 matrix whose entry is the linking of $\beta^1-\beta^2$ with itself, namely $\begin{pmatrix}1\end{pmatrix}$.  Hence the $\sigma(W(\alpha,\beta))=1$.  Since $\beta$ is an unknot with zero self-linking it follows that $|\Xi_3(\alpha)|=1$.
\end{ex}
 
\begin{ex}\label{bigexampleprocedure.ex} We compute $\Xi_3(\alpha)$, where $\alpha=\alpha_{1,1}$ is the first knot in the family $\alpha_{a,b}$ pictured in Figure ~\ref{curves.fig}.  Note that $\alpha_{a,b}$ is the knot $C(2a,2,2b,-2,-2a,2b)$ in Appendix A of \cite{kjuchukova2016classification}, and is one of the infinite families of two-bridge ribbon knots discovered by Casson and Gordon~\cite{lamm2000symmetric}.  By Corollary \ref{twobridge}, we know $\Xi_3(\alpha)=\pm 1$, so our goal is to show this independently using Theorem \ref{procedurethm}. 
	
In this example, unlike the previous one, the curves $\omega-i$ come into play.   A Seifert surface $V$ for $\alpha$, a mod $3$ characteristic knot $\beta$ (see Appendix A of \cite{kjuchukova2016classification} for details), and a choice of curves $\omega_i$ are also shown.  A schematic for a link diagram containing $\alpha$, the $\omega_i$, $\beta_r$, and $\beta_l$ is shown in Figure \ref{bigexamplediagram.fig}.  A few sample anchor paths for the $\omega_i$, $\beta_r$, and $\beta_l$ are shown in Figure ~\ref{curveswithanchorpaths.fig}.

We use the computer program in \cite{cahnkjuchukova2016linking} to find all linking numbers of lifts of the $\omega_i$ and $\beta$.  These linking numbers are displayed in Table \ref{linkingnumbers.tab}.  For curves which intersect on $V$, we make a choice of resolution of the intersection point.  The signature is independent of this choice.

		Computing the monodromies for each anchor path (see the previous example for more details), and applying the rule in Theorem \ref{procedurethm}, we find that $\Xi_3(\alpha_{1,1})$ is the signature of the matrix of linking numbers of $\omega_1^1-\omega_1^2$, $\omega_2^1-\omega_2^2$, $\omega_3^2-\omega_3^3$, $\omega_4^1-\omega_4^2$, and $\beta_1-\beta_2$.  This matrix is 
		$$\begin{pmatrix}-2&2&-2&2&0\\2&-1&3&1&-3\\-2&3&-3&-5&-1\\2&1&-5&-2&0\\0&-3&-1&0&-1\end{pmatrix},$$ which has signature $-1$. As in the previous example, $\beta$ is an uknot with zero self-linking.  Therefore $|\Xi_3(\alpha_{1,1})|=1$.
		\end{ex}

\begin{figure}[h]
	\includegraphics[width=4.6in]{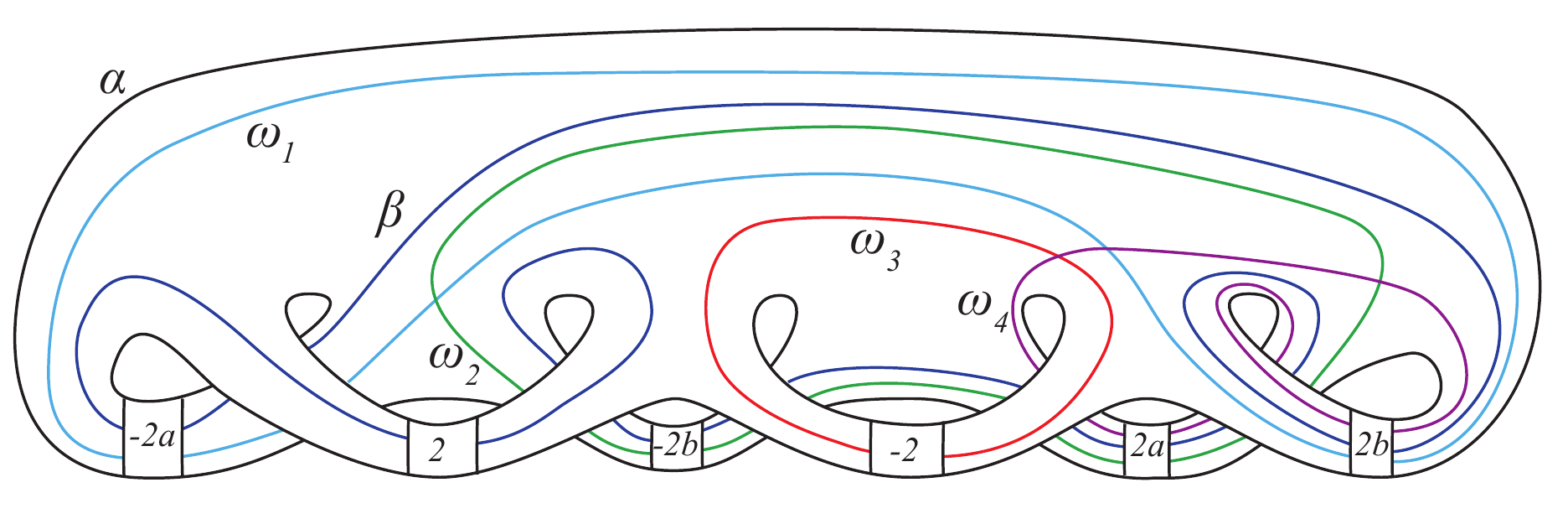}
	\caption{The Seifert surface $V$ of a family of two-bridge slice knots $\alpha(a,b)$, one possible family of characteristic knots $\beta$, and collection of curves which form a basis for $H_1(V-\beta)$.}
	\label{curves.fig}
	\end{figure}
	
	\begin{figure}[h]\includegraphics[width=5in]{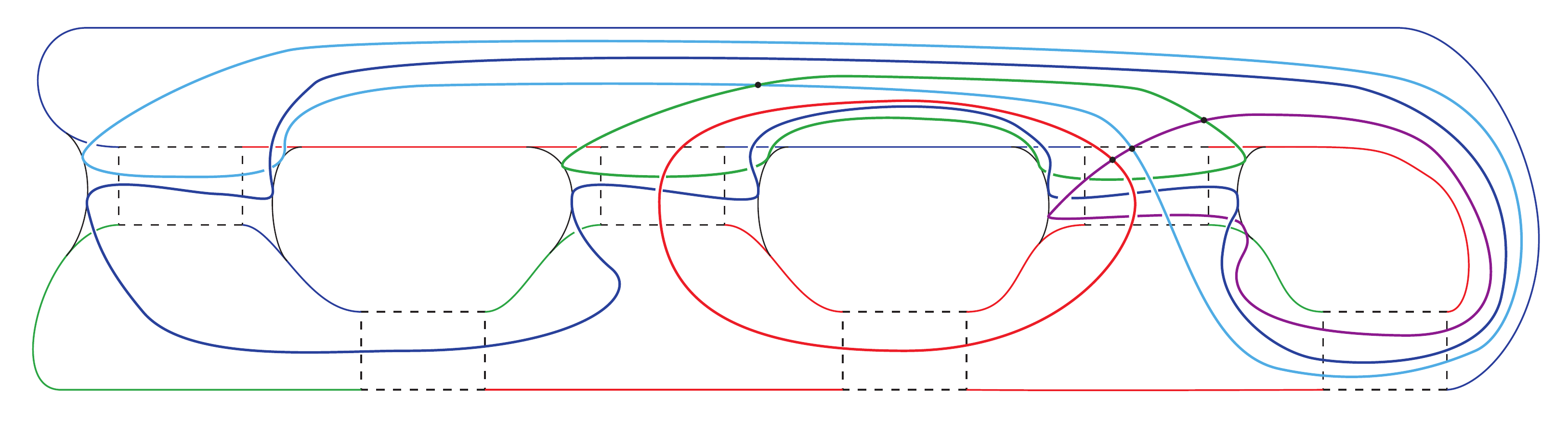}
		\caption{A schematic for drawing the link diagram for $\alpha$ and the curves in Figure \ref{curves.fig}.}
		\label{bigexamplediagram.fig}
		\end{figure}

	\begin{table}[h]
	\tiny
	\centering
	\resizebox{\textwidth}{!}{
	\begin{tabular}{c|c|c|c|c|c|}
		&$\omega_1$&$\omega_2$&$\omega_3$&$\omega_4$&$\beta$\\ \hline
	$\omega_1$	&$\begin{pmatrix} -1&0&0\\0&-1&0\\0&0&-1\end{pmatrix}$&$\begin{pmatrix} 1&0&0\\0&1&0\\0&0&1\end{pmatrix}$&$\begin{pmatrix} 1&-1&0\\-1&1&0\\0&0&0\end{pmatrix}$&$\begin{pmatrix} 1&-1&0\\0&0&0\\-1&1&0\end{pmatrix}$&$\begin{pmatrix}0&0&0\\0&0&0\\0&0&0\end{pmatrix}$\\ \hline
	$\omega_2$&$\begin{pmatrix} 1&0&0\\0&1&0\\0&0&1\end{pmatrix}$&$\begin{pmatrix} 1&1&0\\1&0&1\\0&1&1\end{pmatrix}$&$\begin{pmatrix} 1&1&-1\\0&0&1\\0&0&1\end{pmatrix}$
	&$\begin{pmatrix} 1&1&0\\0&1&1\\1&0&1\end{pmatrix}$&$\begin{pmatrix}-2&-1&0\\1&-1&-1\\0&-1&-2\end{pmatrix}$\\\hline
	
		$\omega_3$&$\begin{pmatrix} 1&-1&0\\-1&1&0\\0&0&0\end{pmatrix}$&$\begin{pmatrix} 1&0&0\\1&0&0\\-1&1&1\end{pmatrix}$
	&$\begin{pmatrix} -3&3&0\\3&-3&0\\0&0&0\end{pmatrix}$&$\begin{pmatrix} -2&3&0\\0&0&1\\3&-2&0\end{pmatrix}$
	&$\begin{pmatrix}0&-1&-2\\-2&-1&0\\-1&-1&-1\end{pmatrix}$\\
		\hline
	$\omega_4$&$\begin{pmatrix} 1&0&-1\\-1&0&1\\0&0&0\end{pmatrix}$&
	$\begin{pmatrix} 0&0&1\\1&0&0\\0&1&0\end{pmatrix}$
	&	$\begin{pmatrix} -2&0&2\\2&0&-2\\0&0&0\end{pmatrix}$
		&$\begin{pmatrix} -2&0&2\\0&0&0\\2&0&-2\end{pmatrix}$&$\begin{pmatrix}0&0&0\\0&0&0\\0&0&0\end{pmatrix}$\\
	\hline
	$\beta$&$\begin{pmatrix}0&0&0\\0&0&0\\0&0&0\end{pmatrix}$&$\begin{pmatrix}-2&1&0\\-1&-1&-1\\0&-1&-2\end{pmatrix}$&$\begin{pmatrix}0&-2&-1\\-1&-1&-1\\-2&0&-1\end{pmatrix}$&$\begin{pmatrix}0&0&0\\0&0&0\\0&0&0\end{pmatrix}$&$\begin{pmatrix} -1& 0& 1\\ 0& 0& 0\\ 1& 0& -1\end{pmatrix}$\\\hline
	\end{tabular}
	}
	\caption{Linking numbers of lifts of the curves $\omega_i$ in Figure~\ref{curves.fig}, with $a=b=1$.     }\label{linkingnumbers.tab}
	\end{table}
		
			\begin{figure}[h]
		\includegraphics[width=4.6in]{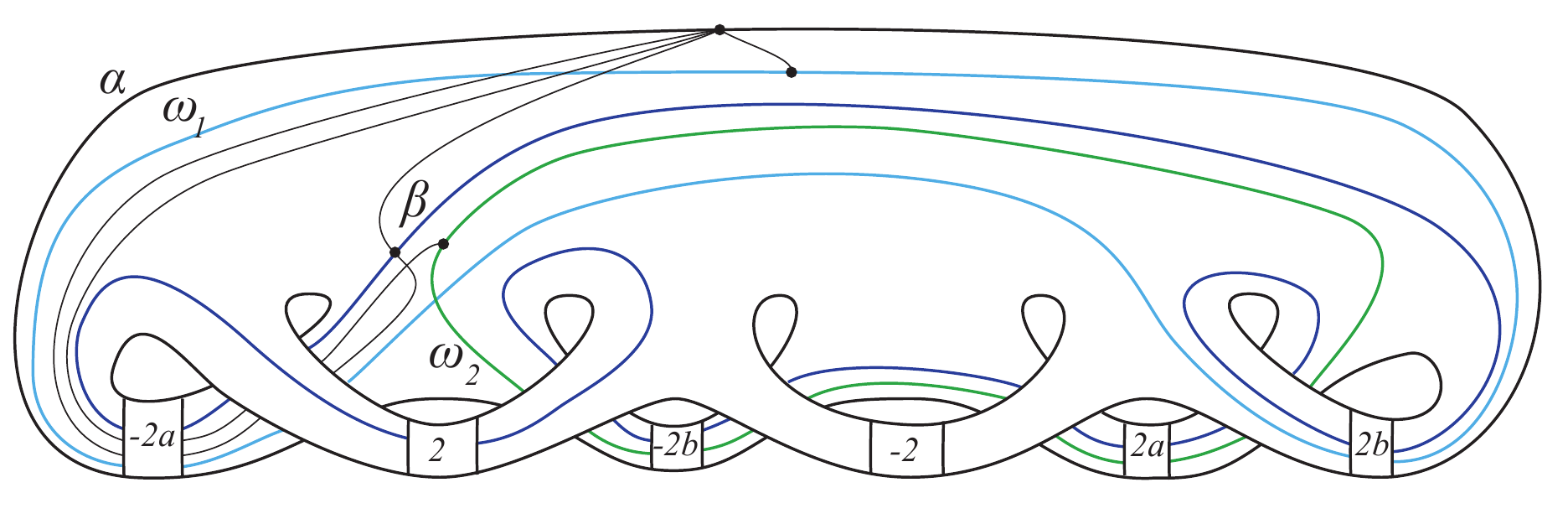}
		\caption{Anchor paths for $\omega_1$, $\omega_2$, and the right and left push-offs of $\beta$.}
		\label{curveswithanchorpaths.fig}
		\end{figure}

		\section{Proof of Theorem~\ref{procedurethm} and the Cappell-Shaneson construction}\label{octopus}
		
Before proving Theorem~\ref{procedurethm} we briefly review the Cappell-Shaneson construction of $M_\alpha$, the irregular dihedral cover of $S^3$ branched along $\alpha$, and a cobordism $W_{\alpha,\beta}$ between $M_\alpha$ and the cyclic cover of $S^3$ along $\beta$, a characteristic knot for $\alpha$.  We again focus on the case $p=3$, but our combinatorial procedure can be generalized for all odd $p$.
	
\subsection{The Capell-Shaneson Construction of the Irregular Dihedral Cover}\label{dih-con}  Let $V$ be a Seifert surface for $\alpha$. Cappell and Shaneson showed that an irregular $p$-fold dihedral branched cover of $S^3$ along $\alpha$ can be obtained from the $p$-fold cyclic branched cover of $S^3$ along a characteristic knot $\beta$ as follows. Roughly speaking, one begins with the $p$ fold cyclic cover $M_\beta$, removes a neighborhood $J$ of the union of the preimages of $V$ from $M_\beta$ to get a 3-manifold with boundary $\partial J$, and identifies points on that boundary via an involution $\bar{h}$ defined below.  The resulting closed manifold is the $p$-fold irregular dihedral cover of $S^3$ branched along $\alpha$.  The surface $S=\bar{h}(\partial J)$ sits inside this covering space, and has boundary equal to the index 1 lift of $\alpha$.  The index 2 lift of $\alpha$ is an embedded curve on $S$.

In order to compute the signature of $Y$, we must compute a matrix of linking numbers of certain elements of $H_1(S;\mathbb{Z})$; namely, a basis for the kernel of the map $i_*:H_1(S; \mathbb{Z})\rightarrow H_1(J;\mathbb{Z})$, where the inclusion $i$ is given by the composition $S=\bar{h}(\partial J)\hookrightarrow \partial J\hookrightarrow J$.

Now we describe the construction in detail and introduce the necessary notation.   Let $f:M_\beta\rightarrow S^3$ be a $3$-fold cyclic covering map branched along $\beta$.  By the construction of Cappell and Shaneson \cite{CS1984linking}, we know that $M_{\alpha}$ can be obtained from $M_\beta$ as follows. Let $J=f^{-1}(V\times[-1,1])\subset M_\beta$. Let $h:V\times[-1,1]\rightarrow V\times[-1,1]$ be given by $h(x,t)=h(x,-t)$.  Let $\bar{h}:\partial J\rightarrow \partial J$ be the lift of $h$ to $M_\beta$ restricted to $\partial J$.  Cappell and Shaneson show that $M_{\alpha}$ is homeomorphic to $(M_\beta-\mathring{J})/\{\bar{h}(x)\sim x \text{ for } x\in \partial J\}$, and that the mapping cone $W_{\alpha,\beta}$ of $\bar{h}$ is a cobordism from the $3$-fold cyclic cover $M_\beta$ to the irregular dihedral cover $M_{\alpha}$.  The surface $S=\bar{h}(\partial J)$ is embedded in $M_{\alpha}$, and has one boundary component $\alpha_0$, the index 1 lift of $\alpha$. 

\begin{figure}[htbp]\includegraphics[width=3in]{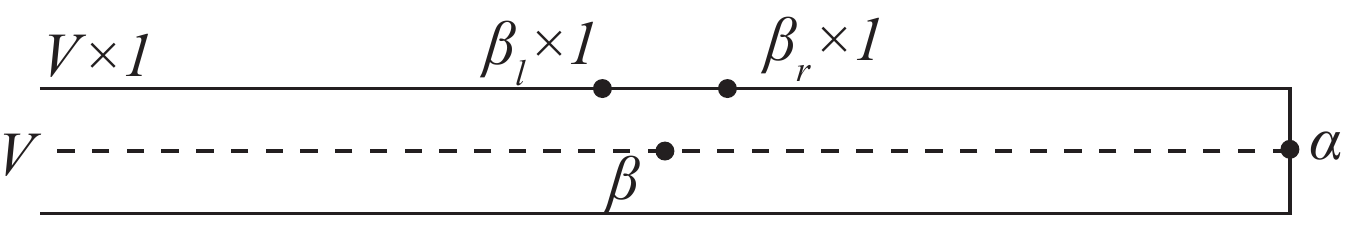}
	\caption{$V$ is a Seifert surface for $\alpha$.  The needed push-offs of $\beta$ in $V\times[-1,1]$ are shown above.}\label{VtimesI.fig}
	\end{figure}

Let $V-\beta$ denote the surface $V$ cut along $\beta$, which we obtain by removing a thin annulus between the right and left push-offs $\beta_r$ and $\beta_l$ of $\beta$ in $V$ (note that $\beta$ is oriented).  More concretely, $S$ can be obtained abstractly by gluing together three copies of $V-\beta$ as follows.  There are three lifts of $(V-\beta)\times 1$ in $M_\beta$, which we label $V_0$, $V_1$, and $V_2$, according to the action of the deck transformation group.  Let $\alpha_0$, $\alpha_1$ and $\alpha_2$ denote the corresponding lifts of $\alpha$. Each $V_i$ contains lifts of the curves $\beta_r\times 1$ and $\beta_l\times 1$, and we denote these by $\beta_{i,r}$ and $\beta_{i,l}$.  See Figures \ref{VtimesI.fig} and \ref{halfstar.fig}.
\begin{figure}[htbp]
	\begin{subfigure}[b]{.4\textwidth}
	\includegraphics[width=2.8in]{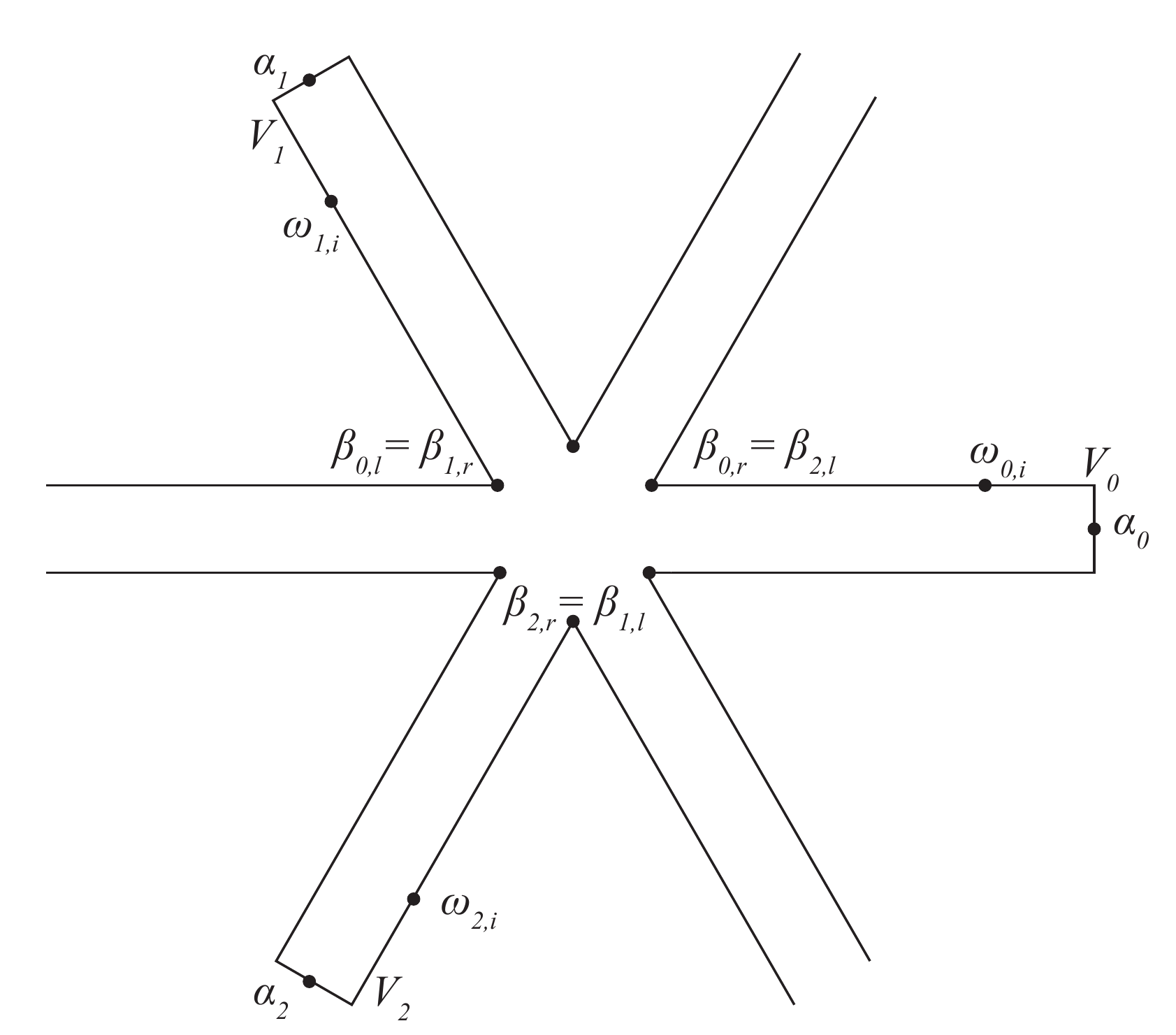}
	\end{subfigure} \quad\quad\quad\raisebox{1in}{$\xrightarrow{\bar{h}}$}\quad
	\begin{subfigure}[b]{.4\textwidth}
	\raisebox{.75in}{\includegraphics[width=2.8in]{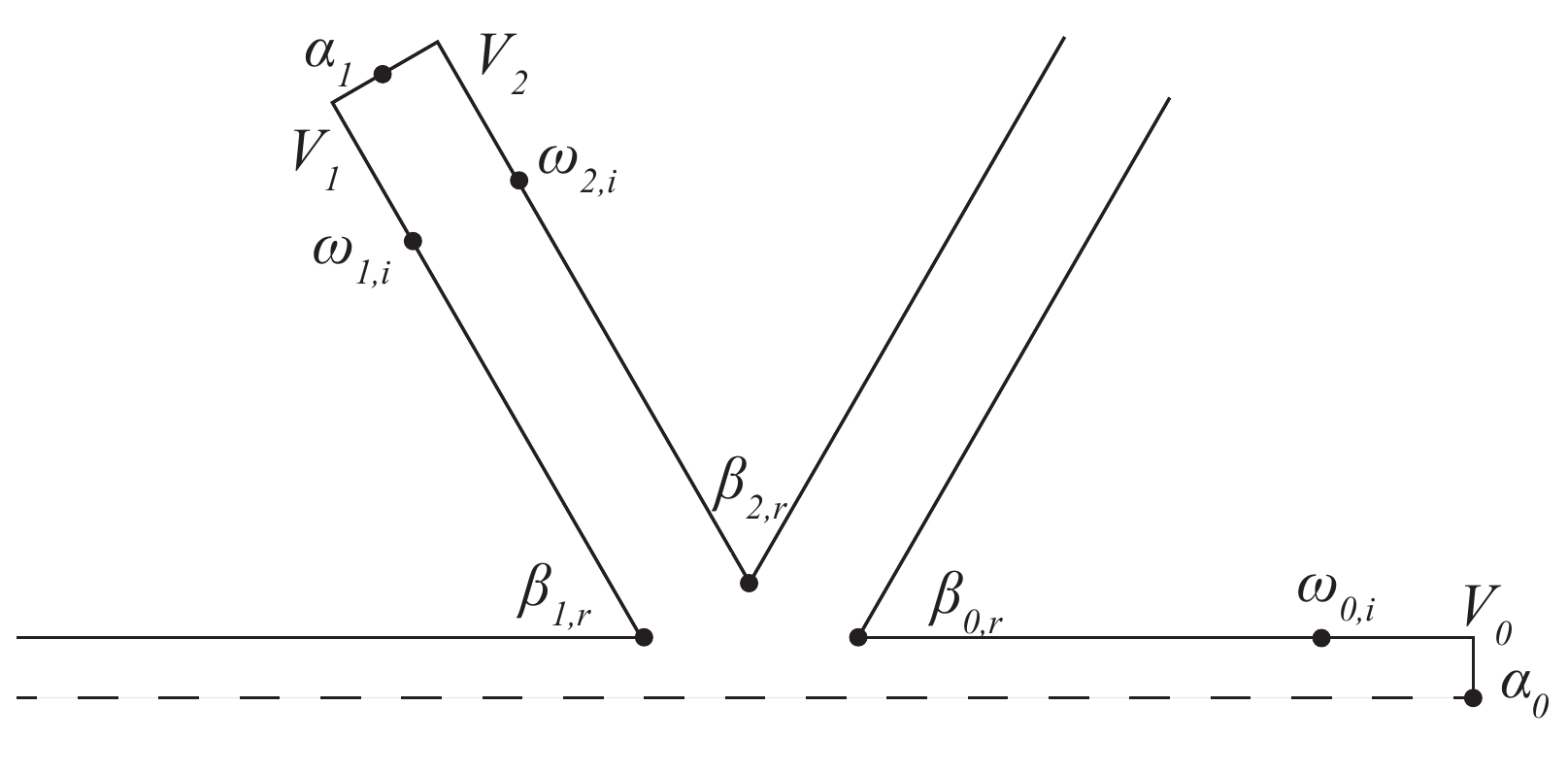}}
	\end{subfigure}
	\caption{The image of $J$ under the involution $\bar{h}$.} \label{halfstar.fig}
	\end{figure}

 From Figure \ref{halfstar.fig}, we can read off the boundaries of the surfaces $V_i$:

$$\partial V_0= \alpha_0 +\beta_{0,r}-\beta_{0,l}$$
$$\partial V_1=\alpha_1 + \beta_{1,r}-\beta_{1,l} $$
$$\partial V_2=\alpha_2 + \beta_{2,r}-\beta_{2,l}.$$

Now we construct $S$ by gluing together $V_0$, $V_1$, and $V_2$ using the following identifications: $\beta_{0,l}$ is identified with $\beta_{1,r}$, $\beta_{1,l}$ is identified with $\beta_{2,r}$, and $\beta_{2,l}$ is identified with $\beta_{0,r}$.  In addition $\alpha_1$ and $\alpha_2$ are identified.  The index 1 and index 2 branch curves are $\alpha_0=\partial S$ and $\alpha_1$ respectively.  Note that $\beta_{0,r}$ and $\beta_{1,r}$ are homologous in $S$, as they cobound $V_0$ together with $\alpha_0$.  The surface $S$, constructed using these identifications, is pictured in Figure \ref{halfoctopus.fig}, in the case where $V$ has genus one and each $V_i$ is a pair of pants.  This is in fact the case in our first example, where $\alpha$ is the knot $6_1$.  In general the genus of $V_i$ is one less than the genus of $V$.

\begin{figure}[htbp]
	\includegraphics[width=2in]{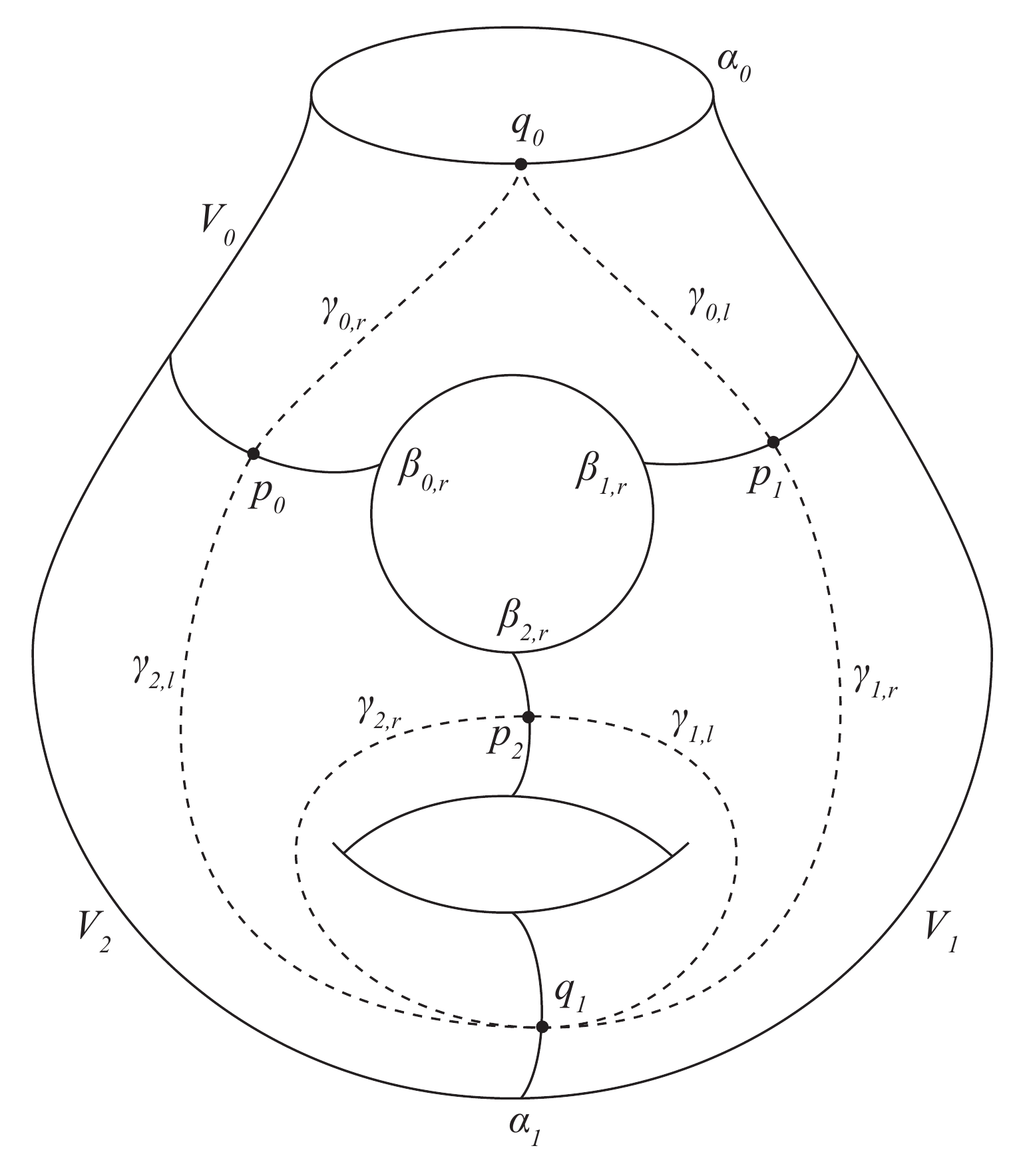}
	\caption{The surface $S$, together with the index one and two branch cuves $\alpha_0$ and $\alpha_1$, the three lifts of $\beta$, and the three lifts of $\gamma_r$ and $\gamma_l$.}
\label{halfoctopus.fig}
	\end{figure}
		
\subsection{Proof of Theorem \ref{procedurethm}} Corollary 2.4 of \cite{kjuchukova2016classification} describes a basis for $\ker i_*$.  The signature defect is the signature of the matrix of linking numbers of elements of $\ker i_*$.  Here we give a geometric description of the elements of $\ker i_*$.  Then we use anchor paths and their monodromies to describe these curves combinatorially using only diagrammatic information, proving Theorem \ref{procedurethm}.
		
		Let $q$ be a point on $\alpha$, and let $p$ be a point on $\beta$. Let $\{\omega_1,\dots,\omega_{2g-2}\}\cup \{\beta_r,\beta_l\}$ be a basis for $H_1(V-\beta;\mathbb{Z})$, where $g$ is the genus of $V$.  Each curve $\omega_i$ in $V-\beta$ has three lifts $\omega_{0,i}$, $\omega_{1,i}$ and $\omega_{2,i}$ to $S\subset M_{\alpha}$. From Figure \ref{halfstar.fig}, it is evident that the differences of curves $\beta_{0,r}- \beta_{2,r}$ and $\omega_{1,i}-\omega_{2,i}$ form a basis for $\ker i_*$. 
		
Now we use anchor paths to describe these curves diagramatically.  Let $\gamma_r$ and $\gamma_l$ be embedded paths from $p$ to $q$ in $V-\beta$ such that the concatenation $\gamma_r\cdot \gamma_l$ is a curve intersecting $\beta$ once transversely, and which completes $\{\beta\}\cup\{\omega_i\}$ to a basis for $H_1(V)$.  The lifts of $\gamma_r$ and $\gamma_l$ to $S\subset M_{\alpha}$ are pictured in Figure \ref{halfoctopus.fig}.  The lifts $\gamma_{0,r}$ and $\gamma_{0,l}$ of $\gamma_r$ and $\gamma_l$ beginning at the lift $q_0$ of $q$ on $\alpha_0$ end on $\beta_{0,r}$ and $\beta_{1,r}$.  

		 If $\delta$ is a path in $V-\beta$ from a point $r$ on $\omega$ to the point $q$ on $\alpha$, then the lift $\delta_0$ of $\delta$ to $V_0$ connects the point $r_0$ on $\omega_{0,i}$ to the point $q_0$ on the index 1 curve $\alpha_0$, while the other lifts $\delta_1$ and $\delta_2$ of $\delta$ connect points $r_1$ and $r_2$ on $\omega_{1,i}$ and $\omega_{2,i}$ to the point $q_1$ on the index two curve $\alpha_1$. See Figure \ref{halfoctopuswithomega.fig}. Reformulating this information in terms of our cell structure yields Theorem \ref{procedurethm}.  \qed
		 
		 \begin{figure}\includegraphics[width=2in]{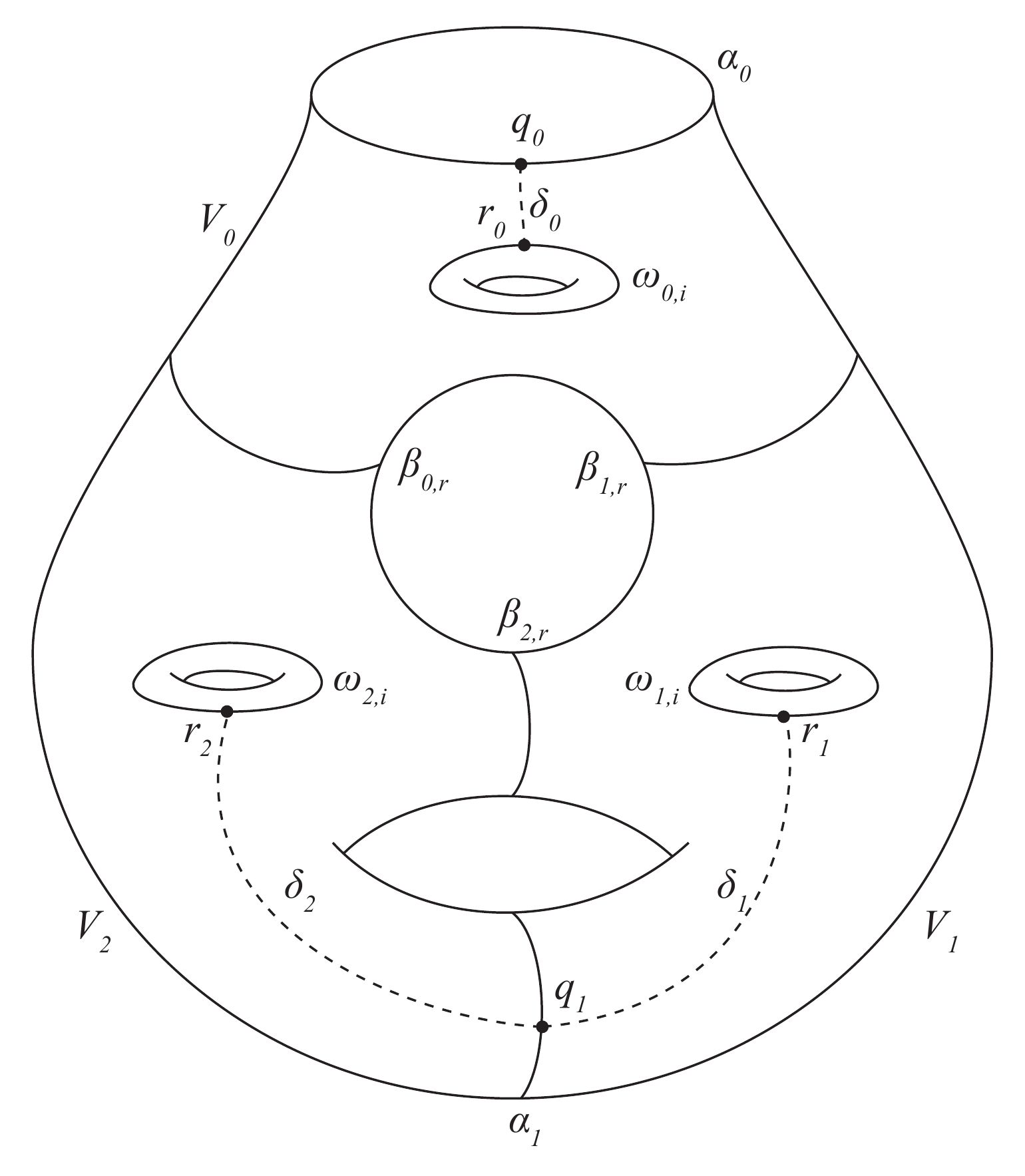}
	 \caption{Lifts to the surface $S$ of an anchor path for $\omega_i$.}
	 \label{halfoctopuswithomega.fig}
	 \end{figure}
 
\section*{Acknowledgment}
	We would like to thank Julius Shaneson and Sebastian Baader for helpful discussions.  

\vspace{.5in}

Patricia Cahn\\
Smith College\\
{\it pcahn@smith.edu}

Alexandra Kjuchukova\\
University of Wisconsin -- Madison\\
{\it kjuchukova@wisc.edu}

\pagebreak
\bibliographystyle{amsplain}
\bibliography{BrCovBib}

\providecommand{\bysame}{\leavevmode\hbox to3em{\hrulefill}\thinspace}
\providecommand{\MR}{\relax\ifhmode\unskip\space\fi MR }
\providecommand{\MRhref}[2]{%
  \href{http://www.ams.org/mathscinet-getitem?mr=#1}{#2}
}
\providecommand{\href}[2]{#2}
\begin{thebibliography}{10}

\bibitem{cahnkjuchukova2016linking}
Patricia Cahn and Alexandra Kjuchukova, \emph{Linking numbers in
  three-manifolds}, arXiv preprint arXiv:1611.10330 (2016).

\bibitem{CS1984linking}
Sylvain Cappell and Julius Shaneson, \emph{Linking numbers in branched covers},
  Contemporary Mathematics \textbf{35} (1984), 165--179.

\bibitem{fox1970metacyclic}
Ralph Fox, \emph{Metacyclic invariants of knots and links}, Canad. J. Math
  \textbf{22} (1970), 193--201.

\bibitem{freedman1982topology}
Michael Freedman, \emph{The topology of four-dimensional manifolds}, Journal of
  Differential Geometry \textbf{17} (1982), no.~3, 357--453.

\bibitem{gaykirby2016trisections}
David Gay and Robion Kirby, \emph{Trisecting 4-manifolds}, Geometry and
  Topology \textbf{20} (2016), no.~6, 3097--3132.

\bibitem{greene2011slice}
Joshua Greene and Stanislav Jabuka, \emph{The slice-ribbon conjecture for
  3-stranded pretzel knots}, American journal of mathematics \textbf{133}
  (2011), no.~3, 555--580.

\bibitem{hilden1974every}
Hugh Hilden, \emph{Every closed orientable 3-manifold is a 3-fold branched
  covering space of ${S}^3$}, Bulletin of the American Mathematical Society
  \textbf{80} (1974), no.~6, 1243--1244.

\bibitem{kjuchukova2016classification}
Alexandra Kjuchukova, \emph{On the classification of irregular dihedral
  branched covers of topological four-manifolds}, arXiv preprint
  arXiv:1608.03329 (2016).

\bibitem{kjorr2017admissible}
Alexandra Kjuchukova and Kent Orr, \emph{Admissible singularities on dihedral
  covers between four-manifolds}, In preparation.

\bibitem{lamm2000symmetric}
Christoph Lamm, \emph{Symmetric unions and ribbon knots}, Osaka Journal of
  Mathematics \textbf{37} (2000), no.~3, 537--550.

\bibitem{lisca2007lens}
Paolo Lisca, \emph{Lens spaces, rational balls and the ribbon conjecture},
  arXiv preprint math/0701610 (2007).

\bibitem{meier2016classification}
Jeffrey Meier, Trent Schirmer, and Alexander Zupan, \emph{Classification of
  trisections and the generalized property r conjecture}, Proceedings of the
  American Mathematical Society \textbf{144} (2016), no.~11, 4983--4997.

\bibitem{meierzupan}
Jeffrey Meier and Alexander Zupan, \emph{Bridge trisections of knotted surfaces
  in $s^4$}.

\bibitem{montesinos1974representation}
Jos{\'e}~Mar{\'\i}a Montesinos, \emph{A representation of closed orientable
  3-manifolds as 3-fold branched coverings of ${S}^3$}, Bulletin of the
  American Mathematical Society \textbf{80} (1974), no.~5, 845--846.

\end{thebibliography}

\end{document}